\documentclass[reqno,11pt]{article} 
\usepackage[left=2.5cm,right=2.5cm,top=3cm,bottom=3cm,a4paper]{geometry}
\usepackage{amsmath, amssymb}
\usepackage{amsthm, amscd} 
\usepackage[all,cmtip]{xy}
\usepackage{float}
\usepackage{caption}
\usepackage{color}

\makeatletter
\DeclareRobustCommand{\rvdots}{%
  \vbox{
    \baselineskip4\p@\lineskiplimit\z@
    \kern-\p@
    \hbox{.}\hbox{.}\hbox{.}
  }}
\makeatother

\newcommand{\Mod}[1]{\ (\textup{mod}\ #1)}
\makeatletter
\def\moverlay{\mathpalette\mov@rlay}
\def\mov@rlay#1#2{\leavevmode\vtop{%
   \baselineskip\z@skip \lineskiplimit-\maxdimen
   \ialign{\hfil$\m@th#1##$\hfil\cr#2\crcr}}}
\newcommand{\charfusion}[3][\mathord]{
    #1{\ifx#1\mathop\vphantom{#2}\fi
        \mathpalette\mov@rlay{#2\cr#3}
      }
    \ifx#1\mathop\expandafter\displaylimits\fi}
\makeatother

\newcommand{\cupdot}{\charfusion[\mathbin]{\cup}{\cdot}}
\theoremstyle{plain} 
\newtheorem{theorem}{\indent\sc Theorem}[section]
\newtheorem{lemma}[theorem]{\indent\sc Lemma}
\newtheorem{corollary}[theorem]{\indent\sc Corollary}
\newtheorem{proposition}[theorem]{\indent\sc Proposition}

\theoremstyle{definition} 
\newtheorem{definition}[theorem]{\indent\sc Definition}

\newtheorem{remark}[theorem]{\indent\sc Remark}

\newtheorem{thmx}{Theorem}

\makeatletter
\def\address#1#2{\begingroup
\noindent\parbox[t]{7.8cm}{%
\small{\scshape\ignorespaces#1}\par\vskip1ex
\noindent\small{\itshape E-mail address}%
\/: #2\par\vskip4ex}\hfill%
\endgroup}%
\makeatother
\title{On some $p$-adic Galois representations and form class groups}
\author{
\textsc{Ho Yun Jung, Ja Kyung Koo, Dong Hwa Shin and Dong Sung Yoon} 
}
\date{} 
\begin{document}

\allowdisplaybreaks

\maketitle

\footnote{ 
2010 \textit{Mathematics Subject Classification}. Primary 11R37; Secondary 11E08, 11F80, 14K22, 11R29.}
\footnote{ 
\textit{Key words and phrases}. Class field theory,  elliptic curves, form class groups, Galois representations.} \footnote{
\thanks{
}
}

\begin{abstract}
Let $K$ be an imaginary quadratic field of discriminant $d_K$
with ring of integers $\mathcal{O}_K$.
When $K$ is different from $\mathbb{Q}(\sqrt{-1})$ and
$\mathbb{Q}(\sqrt{-3})$, we consider
a certain specific model for the elliptic curve $E_K$ with
$j(E_K)=j(\mathcal{O}_K)$ which is defined over $\mathbb{Q}(j(E_K))$.
In this paper, for each positive integer $N$ we compare the extension field of $\mathbb{Q}$
generated by the coordinates of $N$-torsion points on $E_K$ with the ray class field $K_{(N)}$
of $K$ modulo $N\mathcal{O}_K$. By using this result we investigate the image of a $p$-adic Galois representation
attached to $E_K$ for a prime $p$, in terms of class field theory.
Second, we construct the
definite form class group of discriminant $d_K$ and level $N$ which is isomorphic to $\mathrm{Gal}(K_{(N)}/\mathbb{Q})$.

\end{abstract}

\maketitle

\tableofcontents

\section {Introduction}

Let $K$ be an imaginary quadratic field of discriminant $d_K$, and let
$\mathcal{O}_K$ be its ring of integers.
The theory of complex multiplication shows
that the maximal abelian extension $K^\mathrm{ab}$ of $K$
can be generated by the singular values of some modular functions (\cite[Chapter 10]{Lang87}).
Furthermore, through the Shimura reciprocity law
one can connect the class field theory with the theory of modular functions
(\cite{Shimura}).
For a positive integer $N$,
let $K_{(N)}$ denote the ray class field of
$K$ modulo $N\mathcal{O}_K$.
In particular, $K_{(1)}$ is the Hilbert class field $H_K$ of $K$.
We let $W_{K,\,N}$ be a Cartan subgroup of
$\mathrm{GL}_2(\mathbb{Z}/N\mathbb{Z})$
assigned to the $(\mathbb{Z}/N\mathbb{Z})$-algebra
$\mathcal{O}_K/N\mathcal{O}_K$ stated in $\S$4 (\ref{Cartan}).
In \cite{Stevenhagen} Stevenhagen
improved the Shimura reciprocity law and expressed
$\mathrm{Gal}(K_{(N)}/H_K)$
as the quotient group of $W_{K,\,N}$
by the subgroup corresponding to the unit group $\mathcal{O}_K^\times$.
And, his work brings up a natural question
whether there is a $2$-dimensional representation
attached to an elliptic curve with complex multiplication
whose image in $\mathrm{GL}_2(\mathbb{Z}/N\mathbb{Z})$ is related to $W_{K,\,N}$.
\par
On the other hand, there are remarkable works on torsion
subgroups of elliptic curves with
complex multiplication defined over number fields
and on isogenies between such elliptic curves
devoted by Bourdon and Clark (\cite{B-C}, \cite{B-C2}),
Bourdon, Clark and Pollack (\cite{B-C-P}),
Bourdon, Clark and Stankewicz (\cite{B-C-S})
and Clark and Pollack (\cite{C-P}).
While in this paper, when $K\neq\mathbb{Q}(\sqrt{-1}),\,\mathbb{Q}(\sqrt{-3})$,
we shall focus on the elliptic curve given by
\begin{equation}\label{intromodel}
E_K~:~y^2=4x^3-\frac{J(J-1)}{27}x-J\left(\frac{J-1}{27}\right)^2\quad
\textrm{with}~J=\frac{1}{1728}\,j(E_K)=\frac{1}{1728}\,j(\mathcal{O}_K)
\end{equation}
and examine the
extension field $T_N$ of $\mathbb{Q}$ generated by
the coordinates of $N$-torsion points on $E_K$
as follows (Theorem \ref{TN}) $\colon$

\begin{thmx}\label{pre1}
Assume that $K$ is different from $\mathbb{Q}(\sqrt{-1})$
and $\mathbb{Q}(\sqrt{-3})$.
\begin{enumerate}
\item[\textup{(i)}]
If $N=2$ and $d_K\equiv0\Mod{4}$, then
$T_N$ is the maximal real subfield of $K_{(2)}$
with $[K_{(2)}:T_N]=2$.
\item[\textup{(ii)}] If $N=2$ and $d_K\equiv1\Mod{4}$, then
$T_N=K_{(2)}$.
\item[\textup{(iii)}] If $N\geq3$, then
$T_N$ is the extension field of $K_{(N)}$
generated by $Y_{\left[\begin{smallmatrix}0&\frac{1}{N}\end{smallmatrix}\right]}$
with $[T_N:K_{(N)}]\leq2$.
\end{enumerate}
\end{thmx}

Recently, Lozano-Robledo (\cite{Lozano-Robledo})
has classified all possible images of $p$-adic Galois representations attached to elliptic curves $E$ with complex multiplication defined over $\mathbb{Q}(j(E))$,
from which he deduced an analogue of Serre's open image theorem
(\cite{Serre}).
Let $Q_0=x^2+b_Kxy+c_Ky^2$ be the principal form of discriminant $d_K$.
And, for a prime $p$ let $W_{K,\,p^\infty}$
be the inverse limit of $W_{K,\,p^n}$ ($n\geq1$) as a subgroup of $\mathrm{GL}_2(\mathbb{Z}_p)$.
We shall revisit a part of his work
by analyzing a representation of
$\mathrm{Gal}(T_{p^n}/\mathbb{Q}(j(E_K)))$
into $\mathrm{GL}_2(\mathbb{Z}/p^n\mathbb{Z})$ (Theorem \ref{main1}) $\colon$

\begin{thmx}\label{pre2}
Assume that $K$ is different from $\mathbb{Q}(\sqrt{-1})$ and
$\mathbb{Q}(\sqrt{-3})$. Let
\begin{equation*}
\widehat{W}_{K,\,p^\infty}=\left\langle
W_{K,\,p^\infty},\,
\begin{bmatrix}1&\phantom{-}b_K\\0&-1\end{bmatrix}
\right\rangle\quad(\subseteq\mathrm{GL}_2(\mathbb{Z}_p)).
\end{equation*}
Then, there exists a $p$-adic Galois representation
$\rho^{}_{p^\infty}:\mathrm{Gal}(\overline{\mathbb{Q}}/\mathbb{Q}(j(E_K)))\rightarrow\mathrm{GL}_2(\mathbb{Q}_p)$
whose image is a subgroup of $\widehat{W}_{K,\,p^\infty}$ of index $1$ or $2$.
In particular, if $p$ is an odd prime
for which it does not split in $K$ and
$[K_{(p)}:H_K]$ is even, then
the image of $\rho^{}_{p^\infty}$ coincides with
$\widehat{W}_{K,\,p^\infty}$.
\end{thmx}

It is worth noting that Bourdon and
Clark (\cite{B-C}) and Lozano-Robledo (\cite{Lozano-Robledo}) independently showed Theorem \ref{pre2}
in greater generality without
the assumption $K\neq\mathbb{Q}(\sqrt{-1}),\,
\mathbb{Q}(\sqrt{-3})$.
The main difference between their works and this paper is that
we apply the modularity criterion for Siegel functions
(\cite{K-L}, Proposition \ref{modularity}) and
the concept of a Fricke family (Remark \ref{Frickefamily})
to the special model of $E_K$ given in (\ref{intromodel})
in order to attain Theorem \ref{pre1} and Theorem \ref{pre2}.
\par
Next, we want to describe
finite Galois groups in view of form class groups
as Hasse (\cite{Hasse}) and Deuring (\cite{Deuring}) did.
For a negative integer $D$ such that $D\equiv0$ or $1\Mod{4}$,
let $\mathcal{Q}(D)$ be the set of primitive
positive definite binary quadratic forms of discriminant $D$.
The full modular group $\mathrm{SL}_2(\mathbb{Z})$
acts on the set $\mathcal{Q}(D)$ from the right and
gives rise to the proper equivalence $\sim$.
In \textit{Disquisitiones Arithmeticae}
(\cite{Gauss}) Gauss introduced a beautiful composition law
on $\mathcal{C}(D)=\mathcal{Q}(D)/\sim$, and later
Dirichlet (\cite{Dirichlet}) presented a different approach to
the study of composition and genus theory.
Letting $\mathcal{C}(\mathcal{O})$ be the ideal
class group of the order $\mathcal{O}$ of discriminant
$D$ in the imaginary quadratic field $\mathbb{Q}(\sqrt{D})$,
we also have the isomorphism
\begin{equation*}
\begin{array}{lcl}
\mathcal{C}(D)&\rightarrow&\mathcal{C}(O)\\
~~\mathrm{[}Q\mathrm{]} & \mapsto & [[\omega_Q,\,1]]=[\mathbb{Z}\omega_Q+\mathbb{Z}]]
\end{array}
\end{equation*}
where $Q(x,\,y)\in\mathcal{Q}(D)$ and $\omega_Q$ is the zero of $Q(x,\,1)$ in
the complex upper half-plane (\cite[Theorem 7.7]{Cox}).
In 2004, Bhargava (\cite{Bhargava}) derived a general
law of composition on $2\times2\times2$ cubes of integers, from which
he was able to obtain Gauss' composition law on $\mathcal{C}(D)$ as a simple special case.
\par
Let $K$ be an imaginary quadratic field of discriminant $d_K$.
For a positive integer $N$, let
\begin{equation*}
\mathcal{Q}_N(d_K)=\{Q=ax^2+bxy+cy^2\in\mathcal{Q}(d_K)~|~\gcd(a,\,N)=1\}.
\end{equation*}
Recently, Eum, Koo and Shin (\cite{E-K-S}) constructed
the extended form class group $\mathcal{C}_N(d_K)=\mathcal{Q}_N(d_K)/\sim^{}_N$,
where the equivalence relation $\sim^{}_N$ is induced from
the congruence subgroup $\Gamma_1(N)$. They
equipped $\mathcal{C}_N(d_K)$ with
a group structure so that
it is isomorphic to the ray class group $\mathrm{Cl}(N\mathcal{O}_K)$ of $K$ modulo $N\mathcal{O}_K$.
In this paper, we shall consider the set of definite binary quadratic forms
\begin{equation*}
\mathcal{Q}^\pm_N(d_K)=\{Q,\,-Q=(-1)Q~|~Q\in\mathcal{Q}_N(d_K)\}
\end{equation*}
with a naturally extended equivalence relation $\sim^{}_N$,
and set
\begin{equation*}
\mathcal{C}^\pm_N(d_K)=\mathcal{Q}^\pm_N(d_K)/\sim^{}_N=
\{[Q]^{}_N~|~Q\in\mathcal{Q}^\pm_N(d_K)\}.
\end{equation*}
Let $\mathfrak{c}$ be the complex conjugation on $\mathbb{C}$. Improving the result of \cite{E-K-S}
we shall show that $\mathcal{C}^\pm_N(d_K)$ can
be regarded as a group isomorphic to $\mathrm{Gal}(K_{(N)}/\mathbb{Q})$
(Theorem \ref{C+-}) $\colon$

\begin{thmx}\label{pre3}
The set $\mathcal{C}_N^\pm(d_K)$ can be given
a group structure isomorphic to $\mathrm{Gal}(K_{(N)}/\mathbb{Q})$ so that
it contains $\mathcal{C}_N(d_K)$ as a subgroup
and the element $[-Q_0]^{}_N$ corresponds to
$\mathfrak{c}|_{K_{(N)}}$.
\end{thmx}

By $\mathcal{F}_N$ we mean the field of meromorphic modular functions
for the principal congruence subgroup $\Gamma(N)$
whose Fourier coefficients belong to the $N$th cyclotomic field.
We shall define the definite form class invariants
$f([Q]^{}_N)$ for $Q\in\mathcal{Q}^\pm_N(d_K)$ and $f\in\mathcal{F}_N$
(Definition \ref{definvariant}), and
show that these invariants satisfy
a natural transformation rule via
the isomorphism $\phi^\pm_N$ in (\ref{phi+-}) established in the proof of Theorem \ref{pre3}
(Theorem \ref{pmQQ}) $\colon$

\begin{thmx}\label{pre4}
Let $Q,\,Q'\in\mathcal{Q}^\pm_N(d_K)$ and $f\in\mathcal{F}_N$.
If $f([Q]^{}_N)$ is finite, then
\begin{equation*}
\phi^\pm_N([Q']^{}_N)(f([Q]^{}_N))=
f([Q']^{}_N[Q]^{}_N).
\end{equation*}
\end{thmx}

Note that Theorem \ref{pre4} is a generalization of
the transformation formula of Siegel-Ramachandra invariants,
defined by the singular values of Siegel functions,
via the Artin map (\cite{Siegel}
or \cite[Chapter 11]{K-L}).

\section {Elliptic curves with complex multiplication}\label{Sect2}

In this section we shall find models for elliptic curves with complex multiplication
which will be used for Theorem \ref{pre2} on
$p$-adic Galois representations.
\par
For a lattice $\Lambda$ in $\mathbb{C}$, let $E/\mathbb{C}$ be
an elliptic curve that is complex analytically isomorphic to $\mathbb{C}/\Lambda$.
Then $E$ has an affine model in Weierstrass form
\begin{equation}\label{model}
E~:~y^2=4x^3-g^{}_2x-g^{}_3
\end{equation}
where
\begin{equation*}
g^{}_2=g^{}_2(\Lambda)=60\sum_{\omega\in\Lambda\setminus\{0\}}\frac{1}{\omega^4}
\quad\textrm{and}\quad g^{}_3=g^{}_3(\Lambda)=140\sum_{\omega\in\Lambda\setminus\{0\}}\frac{1}{\omega^6}.
\end{equation*}
Furthermore, the map
\begin{equation}\label{isomorphism}
\begin{array}{ccc}
\mathbb{C}/\Lambda&\rightarrow&E(\mathbb{C})\subset\mathbb{P}^2(\mathbb{C})\\
z+\Lambda&\mapsto&[\wp(z;\,\Lambda):\wp'(z;\,\Lambda):1]
\end{array}
\end{equation}
is an isomorphism of complex Lie groups, where
\begin{equation*}
\wp(z;\,\Lambda)=\frac{1}{z^2}+\sum_{\omega\in\Lambda\setminus\{0\}}
\left(
\frac{1}{(z-\omega)^2}-\frac{1}{\omega^2}\right)\quad(z\in\mathbb{C})
\end{equation*}
is the Weierstrass $\wp$-function which
is an even elliptic function for $\Lambda$.
Letting $j=j(E)=j(\Lambda)$, the $j$-invariant of $E$, we set
\begin{equation*}
J=\frac{1}{1728}{j}=\frac{g_2^3}{\Delta}\quad\textrm{with}~
\Delta=\Delta(\Lambda)=g_2^3-27g_3^2~(\neq0).
\end{equation*}
In particular, if $J\neq0,\,1$ and so $g^{}_2,\,g^{}_3\neq0$, then one can obtain by (\ref{model}) and (\ref{isomorphism}) another model and parametrization of $E$
in such a way that
\begin{equation}\label{reparametrization}
\begin{array}{ccc}
\mathbb{C}/\Lambda&\stackrel{\sim}{\rightarrow}&E(\mathbb{C})~:~
\displaystyle y^2=4x^3-\frac{J(J-1)}{27}x-J\left(\frac{J-1}{27}\right)^2\vspace{0.1cm}\\
z+\Lambda&\mapsto&\displaystyle\left[
\frac{g^{}_2g^{}_3}{\Delta}\wp(z;\,\Lambda):\sqrt{\left(
\frac{g^{}_2g^{}_3}{\Delta}\right)^3}\wp'(z;\,\Lambda):1\right].
\end{array}
\end{equation}
Here, we take the principal branch for $\displaystyle\sqrt{\left(\frac{g^{}_2g^{}_3}{\Delta}\right)^3}$.
\par
Let $K$ be an imaginary quadratic field
and
$\mathcal{O}_K$ be its ring of integers.
Let
\begin{equation*}
\tau^{}_K=\left\{
\begin{array}{cl}
\displaystyle\frac{-1+\sqrt{d_K}}{2} & \textrm{if}~d_K\equiv1\Mod{4},\vspace{0.1cm}\\
\displaystyle\frac{\sqrt{d_K}}{2} & \textrm{if}~d_K\equiv0\Mod{4},
\end{array}\right.
\end{equation*}
and so $\mathcal{O}_K=[\tau^{}_K,\,1]=\mathbb{Z}\tau^{}_K+\mathbb{Z}$.
If we put $\Lambda=\mathcal{O}_K$, then $E$ has complex multiplication by $\mathcal{O}_K$
(\cite[Theorem 4.1 in Chapter VI]{Silverman09}).
Moreover, if $K$ is different from $\mathbb{Q}(\sqrt{-1})$ and $\mathbb{Q}(\sqrt{-3})$, then $J(\mathcal{O}_K)\neq0,\,1$ by the homogeneity of
$g^{}_2(\Lambda)$ and $g^{}_3(\Lambda)$ (\cite[p. 193]{Cox}).
In this case, we denote by $E_K$ the elliptic curve
with the model and parametrization described in (\ref{reparametrization}).
For
$\mathbf{v}=\begin{bmatrix}v_1&v_2\end{bmatrix}\in M_{1,\,2}(\mathbb{Q})\setminus
M_{1,\,2}(\mathbb{Z})$, let
\begin{eqnarray}
X_\mathbf{v}&=&\frac{g^{}_2(\mathcal{O}_K)g^{}_3(\mathcal{O}_K)}
{\Delta(\mathcal{O}_K)}\wp(v_1\tau^{}_K+v_2;\,\mathcal{O}_K),\label{Xv}\\
Y_\mathbf{v}&=&\sqrt{\left(
\frac{g^{}_2(\mathcal{O}_K)g^{}_3(\mathcal{O}_K)}{\Delta(\mathcal{O}_K)}\right)^3}\wp'(v_1\tau^{}_K+v_2;\,\mathcal{O}_K).
\label{Yv}
\end{eqnarray}
For convenience, we set
\begin{equation*}
X_\mathbf{v}=0,~
Y_\mathbf{v}=1\quad\textrm{for}~\mathbf{v}\in M_{1,\,2}(\mathbb{Z}).
\end{equation*}
By the fundamental properties of the Weierstrass $\wp$-function,
we get the following lemma.

\begin{lemma}\label{basicproperties}
Let $\mathbf{u},\,\mathbf{v}\in M_{1,\,2}(\mathbb{Q})
\setminus M_{1,\,2}(\mathbb{Z})$.
\begin{enumerate}
\item[\textup{(i)}]
$X_\mathbf{u}=X_\mathbf{v}$ if and only if
$\mathbf{u}\equiv
\pm\mathbf{v}\Mod{M_{1,\,2}(\mathbb{Z})}$.
\item[\textup{(ii)}]
$Y_{-\mathbf{v}}=-Y_\mathbf{v}$ and
$Y_{\mathbf{v}+\mathbf{n}}=Y_\mathbf{v}$ for $\mathbf{n}\in
M_{1,\,2}(\mathbb{Z})$.
\item[\textup{(iii)}]
$Y_\mathbf{v}=0$ if and only if
$2\mathbf{v}\in M_{1,\,2}(\mathbb{Z})$.
\end{enumerate}
\end{lemma}
\begin{proof}
See \cite[$\S$10.A]{Cox}.
\end{proof}

For a positive integer $N$, let $E_K[N]$ be the subgroup of $E_K(\mathbb{C})$
consisting of $N$-torsion points. Then we have
\begin{eqnarray*}
E_K[N]&=&
\{[X_\mathbf{0}:Y_\mathbf{0}:0]\}\cup
\{[X_\mathbf{v}:Y_\mathbf{v}:1]~|~\mathbf{v}\neq\mathbf{0}~
\textrm{and}~Nv_1,\,Nv_2\in\{0,\,1,\,\ldots,\,N-1\}\}\\
&=&\{[X_\mathbf{0}:Y_\mathbf{0}:0]\}\cup\left\{
[X_\mathbf{v}:Y_\mathbf{v}:1]~|~\mathbf{v}\in\frac{1}{N}M_{1,\,2}(\mathbb{Z})\setminus
M_{1,\,2}(\mathbb{Z})\right\}.
\end{eqnarray*}

\section {Modular functions}

We shall introduce some meromorphic modular functions which
will help us examine the extension field of $\mathbb{Q}$ generated by
the coordinates of $N$-torsion points on $E_K$.
\par
The modular group $\mathrm{SL}_2(\mathbb{Z})$ acts on the complex upper half-plane
$\mathbb{H}$ by fractional linear transformations, that is,
\begin{equation*}
\gamma(\tau)=\frac{a\tau+b}{c\tau+d}\quad\left(\textrm{with}~\gamma=
\begin{bmatrix}a&b\\c&d\end{bmatrix}\in\mathrm{SL}_2(\mathbb{Z}),~
\tau\in\mathbb{H}\right).
\end{equation*}
Let $j$ be the elliptic modular function defined on $\mathbb{H}$, namely,
\begin{equation*}
j(\tau)=j([\tau,\,1])\quad(\tau\in\mathbb{H}).
\end{equation*}

\begin{lemma}\label{jLemma}
The map
$\mathrm{SL}_2(\mathbb{Z})\backslash\mathbb{H}\rightarrow\mathbb{C}$
sending an orbit of $\tau$ to $j(\tau)$
is a well-defined bijection.
\end{lemma}
\begin{proof}
See \cite[Theorem 4 in Chapter 3]{Lang87}.
\end{proof}

For $\mathbf{v}=\begin{bmatrix}v_1&v_2\end{bmatrix}\in M_{1,\,2}(\mathbb{Q})
\setminus M_{1,\,2}(\mathbb{Z})$, the Fricke function $f_\mathbf{v}$ on $\mathbb{H}$
is given by
\begin{equation}\label{defFricke}
f_\mathbf{v}(\tau)=-2^73^5\frac{g^{}_2([\tau,\,1])g^{}_3([\tau,\,1])}
{\Delta([\tau,\,1])}\wp(v_1\tau+v_2;\,[\tau,\,1])\quad(\tau\in\mathbb{H}).
\end{equation}
Note that if $\mathbf{u},\,\mathbf{v}\in M_{1,\,2}(\mathbb{Q})
\setminus M_{1,\,2}(\mathbb{Z})$ such that
$\mathbf{u}\equiv\pm\mathbf{v}\Mod{M_{1,\,2}(\mathbb{Z})}$, then
two functions $f_\mathbf{u}$ and $f_\mathbf{v}$ are equal.
For a positive integer $N$, let
\begin{equation*}
\mathcal{F}_N=\left\{
\begin{array}{ll}
\mathbb{Q}(j) & \textrm{if}~N=1,\\
\mathbb{Q}\left(j,\,f_\mathbf{v}~|~\mathbf{v}\in
M_{1,\,2}(\mathbb{Q})\setminus M_{1,\,2}(\mathbb{Z})~\textrm{such that}~
N\mathbf{v}\in M_{1,\,2}(\mathbb{Z})\right) & \textrm{if}~N\geq2.
\end{array}\right.
\end{equation*}
Then $\mathcal{F}_N$ is a Galois extension of $\mathcal{F}_1$
whose Galois group is isomorphic to $\mathrm{GL}_2(\mathbb{Z}/N\mathbb{Z})/\{\pm I_2\}$
(\cite[Theorem 6.6]{Shimura}).
And, $\mathcal{F}_N$ coincides with the field
of meromorphic modular functions for the principal congruence subgroup
\begin{equation*}
\Gamma(N)=\left\{
\alpha\in\mathrm{SL}_2(\mathbb{Z})~|~
\alpha\equiv I_2\Mod{N M_2(\mathbb{Z})}\right\}
\end{equation*}
whose Fourier coefficients belong to the $N$th cyclotomic field
$\mathbb{Q}(\zeta^{}_N)$ with $\zeta^{}_N=e^{2\pi\mathrm{i}/N}$
(\cite[Proposition 6.9]{Shimura}).
\par
We further observe that the group $\mathrm{GL}_2(\mathbb{Z}/N\mathbb{Z})/\{\pm I_2\}$
can be decomposed as
\begin{equation*}
\mathrm{GL}_2(\mathbb{Z}/N\mathbb{Z})/\{\pm I_2\}=
\mathrm{SL}_2(\mathbb{Z}/N\mathbb{Z})/\{\pm I_2\}\cdot
G_N=G_N\cdot\mathrm{SL}_2(\mathbb{Z}/N\mathbb{Z})/\{\pm I_2\}
\end{equation*}
where $G_N$ is the subgroup of $\mathrm{GL}_2(\mathbb{Z}/N\mathbb{Z})/\{\pm I_2\}$ given by
\begin{equation*}
G_N=\left\{\left[\begin{bmatrix}
1&0\\0&d
\end{bmatrix}\right]=
\textrm{the class of}~\begin{bmatrix}
1&0\\0&d
\end{bmatrix}
~\textrm{in}~\mathrm{GL}_2(\mathbb{Z}/N\mathbb{Z})/\{\pm I_2\}~\bigg|~d\in(\mathbb{Z}/N\mathbb{Z})^\times\right\}.
\end{equation*}
The action of $\mathrm{GL}_2(\mathbb{Z}/N\mathbb{Z})/\{\pm I_2\}$
($\simeq\mathrm{Gal}(\mathcal{F}_N/\mathcal{F}_1)$) on $\mathcal{F}_N$
can be explained as follows $\colon$ first, if
$\sigma$ is an element of $\mathrm{SL}_2(\mathbb{Z}/N\mathbb{Z})/\{\pm I_2\}$
obtained from a matrix $\gamma$ in $\mathrm{SL}_2(\mathbb{Z})$, then
its action on $\mathcal{F}_N$ is
\begin{equation}\label{composition}
f^\sigma=f\circ\gamma\quad(f\in\mathcal{F}_N).
\end{equation}
Second, if $\left[\begin{bmatrix}1&0\\0&d\end{bmatrix}\right]$ is an element of $G_N$ with $d\in(\mathbb{Z}/N\mathbb{Z})^\times$
and $f\in\mathcal{F}_N$ has Fourier expansion
\begin{equation*}
f(\tau)=\sum_{n\gg-\infty}c(n)q^{n/N}\quad(c(n)\in\mathbb{Q}(\zeta^{}_N),~
q=e^{2\pi\mathrm{i}\tau})
\end{equation*}
(here, $n\gg-\infty$ means that
there are only finitely many negative integers $n$
in the above summation),
then
\begin{equation*}
f^{\left[\left[\begin{smallmatrix}1&0\\0&d\end{smallmatrix}\right]\right]}(\tau)=
\sum_{n\gg-\infty}c(n)^{\sigma_d}q^{n/N}
\end{equation*}
where $\sigma_d$ is the automorphism of $\mathbb{Q}(\zeta^{}_N)$ given by
$\zeta^{}_N\mapsto\zeta_N^d$.
\par
Meanwhile, for $\mathbf{v}=\begin{bmatrix}v_1 & v_2\end{bmatrix}\in M_{1,\,2}(\mathbb{Q})\setminus
M_{1,\,2}(\mathbb{Z})$ the Siegel function $g_\mathbf{v}$ on $\mathbb{H}$
is given by the infinite product expansion
\begin{equation}\label{infiniteproduct}
g_\mathbf{v}(\tau)=-q^{(1/2)\mathbf{B}_2(v_1)}e^{\pi\mathrm{i}v_2(v_1-1)}
(1-q_z)\prod_{n=1}^\infty(1-q^nq_z)(1-q^nq_z^{-1})\quad(\tau\in\mathbb{H})
\end{equation}
where $q_z=e^{2\pi\mathrm{i}z}$ with
$z=v_1\tau+v_2$ and $\mathbf{B}_2(x)=x^2-x+1/6$
is the second Bernoulli polynomial.
Note that $g_\mathbf{v}$ has neither a zero nor a pole on $\mathbb{H}$.
One can refer to \cite{Siegel} or \cite{K-L}
for further details on Siegel functions.

\begin{proposition}\label{modularity}
For $N\geq2$, let $\{m(\mathbf{v})\}_{\mathbf{v}\in(1/N)M_{1,\,2}(\mathbb{Z})\setminus
M_{1,\,2}(\mathbb{Z})}$ be a family of integers such that $m(\mathbf{v})=0$
except for finitely many $\mathbf{v}$. Then the finite product
\begin{equation*}
\prod_{\mathbf{v}}g_\mathbf{v}^{m(\mathbf{v})}
\end{equation*}
is a meromorphic modular function for $\Gamma(N)$ if
\begin{eqnarray*}
&&\sum_\mathbf{v}m(\mathbf{v})(Nv_1)^2\equiv
\sum_\mathbf{v}m(\mathbf{v})(Nv_2)^2\equiv0\Mod{N\gcd(2,\,N)},
\\
&&\sum_\mathbf{v}m(\mathbf{v})(Nv_1)(Nv_2)\equiv0\Mod{N},\\
&&\sum_\mathbf{v}m(\mathbf{v})\gcd(12,\,N)\equiv0\Mod{12}.
\end{eqnarray*}
\end{proposition}
\begin{proof}
See \cite[Theorems 5.2 and 5.3]{K-L}.
\end{proof}

\begin{remark}\label{Frickefamily}
We say that $\mathbf{v}\in M_{1,\,2}(\mathbb{Q})$ is primitive modulo $N$ ($\geq2$) if
$N$ is the smallest positive integer so that $N\mathbf{v}\in M_{1,\,2}(\mathbb{Z})$.
Let $V_N$ be the set of all such $\mathbf{v}$'s.
As mentioned in \cite[p. 33]{K-L}
a family $\{h_\mathbf{v}\}_{\mathbf{v}\in V_N}$ of functions in $\mathcal{F}_N$
is called a Fricke family of level $N$ if
\begin{enumerate}
\item[(i)] $h_\mathbf{v}$ is holomorphic on $\mathbb{H}$ for every $\mathbf{v}\in V_N$,
\item[(ii)] $h_\mathbf{u}=h_\mathbf{v}$ whenever $\mathbf{u},\,\mathbf{v}\in V_N$
satisfy $\mathbf{u}\equiv\pm\mathbf{v}\Mod{M_{1,\,2}(\mathbb{Z})}$,
\item[(iii)] $h_\mathbf{v}^\gamma=h_{\mathbf{v}\gamma}$ for all
$\mathbf{v}\in V_N$ and $\gamma\in\mathrm{GL}_2(\mathbb{Z}/N\mathbb{Z})/\{\pm I_2\}$
($\simeq\mathrm{Gal}(\mathcal{F}_N/\mathcal{F}_1)$).
\end{enumerate}
As is well known, $\{f_\mathbf{v}\}_{\mathbf{v}\in V_N}$
and $\{g_\mathbf{v}^{12N}\}_{\mathbf{v}\in V_N}$ are
Fricke families of level $N$
(\cite[Theorem 6.6]{Shimura} and \cite[Proposition 1.3 in Chapter 2]{K-L}).
In \cite{J-K-S17} and \cite{J-K-S19} Jung, Koo and Shin
examined several properties and applications
to class field theory of these typical examples of a Fricke family.
\end{remark}

\begin{lemma}\label{pg}
If $\mathbf{v}=\begin{bmatrix}v_1&v_2\end{bmatrix}\in M_{1,\,2}(\mathbb{Q})$
such that $2\mathbf{v}\not\in M_{1,\,2}(\mathbb{Z})$, then
\begin{equation*}
\wp'(v_1\tau+v_2;\,[\tau,\,1])=
-\eta(\tau)^6\frac{g_{2\mathbf{v}}(\tau)}
{g_\mathbf{v}(\tau)^4}\quad(\tau\in\mathbb{H})
\end{equation*}
where $\eta$ is the Dedekind $\eta$-function on $\mathbb{H}$ defined by
\begin{equation*}
\eta(\tau)=\sqrt{2\pi}\zeta^{}_8q^{1/24}\prod_{n=1}
^\infty(1-q^n)\quad(\tau\in\mathbb{H}).
\end{equation*}
\end{lemma}
\begin{proof}
See \cite[p. 852]{K-S}.
\end{proof}

\section {The Shimura reciprocity law}\label{Sect4}

Let $K$ be an imaginary quadratic field of discriminant $d_K$ ($<0$).
In this section, we shall introduce a description of $\mathrm{Gal}(K_{(p^\infty)}/H_K)$ for a prime $p$
due to Stevenhagen (\cite{Stevenhagen}).
\par
By the classical theory of complex multiplication
established by Kronecker, Weber, Hasse and Deuring, we have
\begin{eqnarray}
&&H_K=K(j(\mathcal{O}_K))=K(j(\tau^{}_K)),\label{generationHilbert}\\
&&K_{(N)}=K(f(\tau^{}_K)~|~f\in\mathcal{F}_N~\textrm{is finite at}~\tau^{}_K)
\label{generationray}
\end{eqnarray}
for a positive integer $N$
(\cite[Theorem 1 and Corollary to Theorem 2 in Chapter 10]{Lang87}). Shimura revisited
this result by developing the theory of canonical models for modular curves and the reciprocity law
(\cite[Theorem 6.31 and Proposition 6.33]{Shimura}).
Let $\min(\tau^{}_K,\,\mathbb{Q})=x^2+b_Kx+c_K$ ($\in\mathbb{Z}[x]$), and let
\begin{equation}\label{Cartan}
W_{K,\,N}=\left\{
\gamma=\begin{bmatrix}t-b_Ks & -c_Ks\\s&t\end{bmatrix}~|~s,\,t\in\mathbb{Z}/N\mathbb{Z}~
\textrm{such that}~\gamma\in\mathrm{GL}_2(\mathbb{Z}/N\mathbb{Z})\right\}
\end{equation}
which is the Cartan subgroup of $\mathrm{GL}_2(\mathbb{Z}/N\mathbb{Z})$
associated with the $(\mathbb{Z}/N\mathbb{Z})$-algebra $\mathcal{O}_K/N\mathcal{O}_K$
with ordered basis $\{\tau^{}_K+N\mathcal{O}_K,\,1+N\mathcal{O}_K\}$.
And, let
\begin{equation*}
U_K=\left\{\begin{array}{ll}
\{\pm I_2\} & \textrm{if}~K\neq\mathbb{Q}(\sqrt{-1}),\,\mathbb{Q}(\sqrt{-3}),\\
\left\{\pm I_2,\,\pm\left[\begin{smallmatrix}
0 & -1\\1 &\phantom{-}0
\end{smallmatrix}\right]\right\}& \textrm{if}~K=\mathbb{Q}(\sqrt{-1}),\\
\left\{\pm I_2,\,\pm\left[\begin{smallmatrix}
-1 & -1\\\phantom{-}1 & \phantom{-}0
\end{smallmatrix}\right],\,
\pm\left[\begin{smallmatrix}
0 & -1\\1&\phantom{-}1\end{smallmatrix}\right]\right\}& \textrm{if}~K=\mathbb{Q}(\sqrt{-3})
\end{array}\right.
\end{equation*}
which is a subgroup of $\mathrm{SL}_2(\mathbb{Z})$.
If
\begin{equation*}
r^{}_N~:~\mathrm{SL}_2(\mathbb{Z})~\rightarrow~\mathrm{SL}_2(\mathbb{Z}/N\mathbb{Z})
\end{equation*}
denotes the reduction modulo $N$, then
$r^{}_N(U_K)$ is a subgroup of $W_{K,\,N}$.
The following proposition done by Stevenhagen gives a simple description of
a part of the Shimura reciprocity law.

\begin{proposition}\label{reciprocity}
We have a surjection
\begin{equation}\label{W}
\begin{array}{ccl}
W_{K,\,N}&\rightarrow&\mathrm{Gal}(K_{(N)}/H_K)\\
\gamma&\mapsto&\left(
f(\tau^{}_K)\mapsto f^{[\gamma]}(\tau^{}_K)~|~f\in\mathcal{F}_N~
\textrm{is finite at}~\tau^{}_K\right)
\end{array}
\end{equation}
whose kernel is $r^{}_N(U_K)$. Here,
$[\gamma]$ means the image of $\gamma$ in $\mathrm{GL}_2(\mathbb{Z}/N\mathbb{Z})/\{\pm I_2\}$
\textup{(}$\simeq\mathrm{Gal}(\mathcal{F}_N/\mathcal{F}_1)$\textup{)}.
\end{proposition}
\begin{proof}
See \cite[$\S$3]{Stevenhagen}.
\end{proof}

We let
\begin{equation*}
\psi^{}_N~:~W_{K,\,N}~\rightarrow~\mathrm{Gal}(K_{(N)}/H_K)
\end{equation*}
be the surjection introduced in (\ref{W}).
For each prime $p$, let
\begin{equation*}
W_{K,\,p^\infty}=
\left\{
\gamma=\begin{bmatrix}
t-b_Ks & -c_Ks\\s&t
\end{bmatrix}~|~s,\,t\in\mathbb{Z}_p~
\textrm{such that}~\gamma\in\mathrm{GL}_2(\mathbb{Z}_p)\right\}
\end{equation*}
which is a subgroup of $\mathrm{GL}_2(\mathbb{Z}_p)$.
The following corollary was shown in
\cite[(3.2) and (4.3)]{Stevenhagen}, however,
we shall briefly give its proof in order to compare
with Theorem \ref{main1}.

\begin{corollary}\label{Wp}
Let $p$ be a prime,
and let $K_{(p^\infty)}$
be the maximal abelian extension of $K$
unramified outside prime ideals lying above $p$. Then $\mathrm{Gal}(K_{(p^\infty)}/H_K)$ is isomorphic to
$W_{K,\,p^\infty}/U_K$.
\end{corollary}
\begin{proof}
By the fact that $\mathrm{Gal}(\mathcal{F}_{p^n}/\mathcal{F}_1)
\simeq\mathrm{GL}_2(\mathbb{Z}/p^n\mathbb{Z})/\{\pm I_2\}$ ($n\geq1$) and
Proposition \ref{reciprocity},
we derive
the inverse system of short exact sequences as in Figure \ref{exacts}
in which the first and second vertical maps are reductions and
the third vertical maps are restrictions.

\begin{figure}[h]
\begin{equation*}
\xymatrixcolsep{2pc}\xymatrix{
 & \rvdots \ar[d] & \rvdots \ar[d] & \rvdots \ar[d]&  \\
0 \ar[r]  & r^{}_{p^n}(U_K) \ar[r] \ar[d] & W_{K,\,p^n} \ar[r]^{\hspace{-0.6cm}\psi^{}_{p^n}} \ar[d] & \mathrm{Gal}(K_{(p^{n})}/H_K) \ar[r] \ar[d] & 0\phantom{~,} \\
  & \rvdots \ar[d] & \rvdots \ar[d] & \rvdots \ar[d]&  \\
0 \ar[r]  & r^{}_{p^2}(U_K) \ar[r] \ar[d] & W_{K,\,p^2} \ar[r]^{\hspace{-0.6cm}\psi^{}_{p^2}} \ar[d] & \mathrm{Gal}(K_{(p^2)}/H_K) \ar[r] \ar[d] & 0\phantom{~,} \\
0 \ar[r]  & r^{}_p(U_K) \ar[r]  & W_{K,\,p} \ar[r]^{\hspace{-0.6cm}\psi^{}_p}  & \mathrm{Gal}(K_{(p)}/H_K) \ar[r]  & 0~,
}
\end{equation*}
\caption{An inverse system of short exact sequences}
\label{exacts}
\end{figure}

\noindent
Since the inverse system $\{r^{}_{p^n}(U_K)\}_{n\geq1}$ satisfies
the Mittag-Leffler condition (\cite[p. 164]{Lang02}), we get the exact sequence
\begin{equation*}
\xymatrixcolsep{2pc}\xymatrix{
0 \ar[r]  & \displaystyle\varprojlim_{n\geq1}r^{}_{p^n}(U_K) \ar[r]
& \displaystyle\varprojlim_{n\geq1}W_{K,\,p^n} \ar[r]
& \displaystyle\varprojlim_{n\geq1}\mathrm{Gal}(K_{(p^{n})}/H_K) \ar[r]  & 0
}
\end{equation*}
(\cite[Proposition 10.3]{Lang02}).
Moreover, since
\begin{eqnarray*}
&&\varprojlim_{n\geq1}W_{K,\,p^n}
~\simeq~W_{K,\,p^\infty},\\
&&\varprojlim_{n\geq1}r^{}_{p^n}(U_K)~\simeq~ U_K~(\hookrightarrow W_{K,\,p^\infty}),\\
&&\varprojlim_{n\geq1}\mathrm{Gal}(K_{(p^{n})}/H_K)
~\simeq~\mathrm{Gal}(\bigcup_{n\geq1}K_{(p^n)}/H_K)~=~
\mathrm{Gal}(K_{(p^\infty)}/H_K),
\end{eqnarray*}
we conclude that
$\mathrm{Gal}(K_{(p^\infty)}/H_K)$ is isomorphic to $W_{K,\,p^\infty}/U_K$.
\end{proof}

\begin{proposition}\label{generatorX}
Assume that $K\neq\mathbb{Q}(\sqrt{-1}),\,\mathbb{Q}(\sqrt{-3})$
and $N\geq2$.
\begin{enumerate}
\item[\textup{(i)}] We have
\begin{equation*}
K_{(N)}=H_K\left(X_{\left[\begin{smallmatrix}
0&\frac{1}{N}
\end{smallmatrix}\right]}\right)=
H_K\left(X_\mathbf{v}~|~
\mathbf{v}\in\frac{1}{N}M_{1,\,2}(\mathbb{Z})\setminus
M_{1,\,2}(\mathbb{Z})\right)
\end{equation*}
where $X_\mathbf{v}$ is the $x$-coordinate
function defined in \textup{$\S$\ref{Sect2} (\ref{Xv})}.
\item[\textup{(ii)}] We find that
\begin{equation*}
X_\mathbf{v}^{\psi^{}_N(\gamma)}=X_{\mathbf{v}\gamma}
\quad\left(\mathbf{v}\in\frac{1}{N}M_{1,\,2}(\mathbb{Z})\setminus
M_{1,\,2}(\mathbb{Z}),~\gamma\in W_{K,\,N}\right).
\end{equation*}
\end{enumerate}
\end{proposition}
\begin{proof}
We note by the definition (\ref{defFricke}) of
a Fricke function that
\begin{equation}\label{Xf}
X_\mathbf{v}=-\frac{1}{2^73^5}f_\mathbf{v}(\tau^{}_K)
\quad\left(\mathbf{v}\in\frac{1}{N}M_{1,\,2}(\mathbb{Z})\setminus
M_{1,\,2}(\mathbb{Z})\right).
\end{equation}
\begin{enumerate}
\item[\textup{(i)}] For the first equality, see (\ref{generationHilbert}) and \cite[Corollary to Theorem 7 in Chapter 10]{Lang87}.
The second equality follows from (\ref{generationray}) and the fact $f_\mathbf{v}\in\mathcal{F}_N$.
\item[\textup{(ii)}] This follows from Remark \ref{Frickefamily} and Proposition
\ref{reciprocity}.
\end{enumerate}
\end{proof}

\begin{remark}\label{Hasse-Ramachandra}
\begin{enumerate}
\item[(i)] Proposition \ref{generatorX} (i)
is a consequence of the main theorem of
the theory of complex multiplication.
By using the Kronecker limit formula and
some arithmetic properties of Fricke functions,
Jung, Koo, Shin and Yoon (\cite{J-K-S18} and \cite{K-S-Y}) improved
Proposition \ref{generatorX} (i) so that
\begin{equation*}
K_{(N)}=K\left(X_{\left[\begin{smallmatrix}0&\frac{1}{N}\end{smallmatrix}\right]}\right)
\quad\textrm{or}\quad
K\left(X_{\left[\begin{smallmatrix}0&\frac{2}{N}\end{smallmatrix}\right]}\right),
\end{equation*}
which gives a partial answer to the problem of Hasse and Ramachandra
(\cite[p. 91]{F-L-R} and \cite[p. 105]{Ramachandra}).
\item[(ii)] Let $K$ be an arbitrary imaginary quadratic field,
and let $E$ be the elliptic curve given by the
Weierstrass equation
\begin{equation*}
E~:~y^2=4x^3-g^{}_2x-g^{}_3
\end{equation*}
where $g^{}_2=g^{}_2(\mathcal{O}_K)$ and $g^{}_3=g^{}_3(\mathcal{O}_K)$.
The Weber function $h:E\rightarrow\mathbb{P}^1(\mathbb{C})$
is defined by
\begin{equation*}
h(\varphi(z))=
\left\{\begin{array}{rl}
\displaystyle\frac{g^{}_2g^{}_3}{\Delta}\,x
& \textrm{if}~K\neq\mathbb{Q}(\sqrt{-1}),\,\mathbb{Q}(\sqrt{-3}),\vspace{0.1cm}\\
\displaystyle\frac{g_2^2}{\Delta}\,x^2
& \textrm{if}~K=\mathbb{Q}(\sqrt{-1}),\vspace{0.1cm}\\
\displaystyle\frac{g_3}{\Delta}\,x^3
& \textrm{if}~K=\mathbb{Q}(\sqrt{-3})
\end{array}\right.\quad(z\in\mathbb{C})
\end{equation*}
where $\varphi:\mathbb{C}/\mathcal{O}_K\rightarrow E$
is the parametrization given in (\ref{isomorphism}) and $\Delta=g_2^3-27g_3^2$.
This function gives rise to an isomorphism of the quotient variety
$E/\mathrm{Aut}(E)$ onto $\mathbb{P}^1(\mathbb{C})$
(\cite[Theorem 7 in Chapter 1]{Lang87}).
Originally, \cite[Corollary to Theorem 7 in Chapter 10]{Lang87}
states
\begin{equation*}
K_{(N)}=H_K\left(h(\varphi(z))~|~z\in\mathbb{C}\setminus\mathcal{O}_K~
\textrm{satisfies}~Nz\in\mathcal{O}_K\right).
\end{equation*}
We note that if $K$ is different from $\mathbb{Q}(\sqrt{-1})$
and $\mathbb{Q}(\sqrt{-3})$, then
\begin{equation*}
X_\mathbf{v}=
h(\varphi(v_1\tau^{}_K+v_2))
\quad\left(\mathbf{v}
=\begin{bmatrix}v_1 & v_2\end{bmatrix}\in\frac{1}{N}M_{1,\,2}(\mathbb{Z})\setminus
M_{1,\,2}(\mathbb{Z})\right).
\end{equation*}
\end{enumerate}
\end{remark}

\section {The field generated by $N$-torsion points}

Unless otherwise specified we assume that
$K$ is an imaginary quadratic field other than $\mathbb{Q}(\sqrt{-1})$
and $\mathbb{Q}(\sqrt{-3})$, and $N$ is a positive integer such that $N\geq2$.
Let $\mathbb{Q}(E_K[N])$ be the extension field of $\mathbb{Q}$
generated by the coordinates of $N$-torsion points on $E_K[N]$, namely,
\begin{equation*}
\mathbb{Q}(E_K[N])=\mathbb{Q}\left(
X_\mathbf{v},\,Y_\mathbf{v}~|~
\mathbf{v}\in\frac{1}{N}M_{1,\,2}(\mathbb{Z})\setminus
M_{1,\,2}(\mathbb{Z})\right).
\end{equation*}
We shall describe the field
$\mathbb{Q}(E_K[N])$ by comparing it with the ray class field $K_{(N)}$
(Theorem \ref{pre1}).

\begin{lemma}\label{K-Lconjugation}
Let $\{h_\mathbf{v}\}_{\mathbf{v}\in V_N}$ be a Fricke family
defined in \textup{Remark \ref{Frickefamily}}.
For each $\mathbf{v}\in V_N$,
we have
\begin{equation*}
\overline{h_\mathbf{v}(\tau^{}_K)}=
h_{\mathbf{v}\left[\begin{smallmatrix}1&\phantom{-}b_K\\0&-1\end{smallmatrix}\right]}(\tau^{}_K)
\end{equation*}
where $\overline{\,\cdot\,}$ means the complex conjugation.
\end{lemma}
\begin{proof}
See \cite[Proposition 1.4 in Chapter 1]{K-L}.
\end{proof}

Recall the relation (\ref{Xf}) and the fact
that $\{f_\mathbf{v}\}_{\mathbf{v}\in V_N}$ is a Fricke family of level $N$
as mentioned in Remark \ref{Frickefamily}.

\begin{lemma}\label{real}
Let $R_N=\displaystyle\mathbb{Q}\left(X_\mathbf{v}~|~
\mathbf{v}\in\frac{1}{N}M_{1,\,2}(\mathbb{Z})\setminus M_{1,\,2}(\mathbb{Z})\right)$.
\begin{enumerate}
\item[\textup{(i)}] If $N=2$ and $d_K\equiv0\Mod{4}$, then
$R_N$ is the maximal real subfield of $K_{(2)}$ with
$[K_{(2)}:R_N]=2$.
\item[\textup{(ii)}] Otherwise, if $N\geq3$ or $d_K\equiv1\Mod{4}$, then  $R_N=K_{(N)}$.
\end{enumerate}
\end{lemma}
\begin{proof}
\begin{enumerate}
\item[(i)] If $N=2$ and $d_K\equiv0\Mod{4}$, then we have
$b_K=0$ and
\begin{eqnarray*}
\overline{X}_\mathbf{v}&=&X_{\mathbf{v}\left[
\begin{smallmatrix}1&\phantom{-}0\\0&-1\end{smallmatrix}\right]}
\quad\textrm{by Lemma \ref{K-Lconjugation}}\quad
\left(\mathbf{v}\in\frac{1}{2}M_{1,\,2}(\mathbf{v})
\setminus M_{1,\,2}(\mathbb{Z})\right)\\
&=&X_\mathbf{v}\quad\textrm{by Lemma \ref{basicproperties} (i)}.
\end{eqnarray*}
This shows that $R_N$ is contained in $\mathbb{R}$. Moreover, we get by
Remark \ref{Hasse-Ramachandra} (i) that
$KR_N=K_{(2)}$. Thus we obtain that $[K_{(2)}:R_N]=2$ and
$R_N$ is the maximal real subfield of $K_{(2)}$.
\item[(ii)] Now, we assume that $N\geq3$ or $d_K\equiv1\Mod{4}$.
We derive by Lemma \ref{K-Lconjugation} and Lemma \ref{basicproperties} (i) that
if $N\geq3$ and $d_K\equiv0\Mod{4}$, then $b_K=0$ and
\begin{equation*}
\overline{X}_{\left[\begin{smallmatrix}\frac{1}{N}&\frac{1}{N}\end{smallmatrix}\right]}=
X_{\left[\begin{smallmatrix}\frac{1}{N}&-\frac{1}{N}\end{smallmatrix}\right]}
\neq X_{\left[\begin{smallmatrix}\frac{1}{N}&\frac{1}{N}\end{smallmatrix}\right]},
\end{equation*}
and if $d_K\equiv1\Mod{4}$, then
$b_K=1$ and
\begin{equation*}
\overline{X}_{\left[\begin{smallmatrix}\frac{1}{N}&0\end{smallmatrix}\right]}=
X_{\left[\begin{smallmatrix}\frac{1}{N}&\frac{1}{N}\end{smallmatrix}\right]}
\neq X_{\left[\begin{smallmatrix}\frac{1}{N}&0\end{smallmatrix}\right]}.
\end{equation*}
This observation holds that
\begin{equation}\label{RNR}
R_N\not\subseteq\mathbb{R}.
\end{equation}
On the other hand, we note by Lemma \ref{K-Lconjugation} that
\begin{equation}\label{XtR}
X_{\left[\begin{smallmatrix}0 & \frac{t}{N}\end{smallmatrix}\right]}\in\mathbb{R}
\quad\textrm{for all $t\in\mathbb{Z}$ such that $t\not\equiv0\Mod{N}$}.
\end{equation}
We then find that
\begin{eqnarray*}
&&K\left(
X_{\left[\begin{smallmatrix}0 & \frac{t}{N}\end{smallmatrix}\right]}~|~
t\in\mathbb{Z}~\textrm{satisfies}~t\not\equiv0\Mod{N}
\right)\\&=&
K_{(N)}\quad
\textrm{by Proposition \ref{generatorX} (i) and Remark \ref{Hasse-Ramachandra} (i)}\\
&\supseteq& R_N\quad\textrm{again by Proposition \ref{generatorX} (i)}\\
&\supsetneq&R_N\cap\mathbb{R}\quad\textrm{by (\ref{RNR})}\\
&\supseteq&\mathbb{Q}\left(
X_{\left[\begin{smallmatrix}0 & \frac{t}{N}\end{smallmatrix}\right]}~|~
t\in\mathbb{Z}~\textrm{satisfies}~t\not\equiv0\Mod{N}
\right)\quad\textrm{by (\ref{XtR})}.
\end{eqnarray*}
Therefore we conclude by the fact $[K:\mathbb{Q}]=2$ that $R_N=K_{(N)}$.
\end{enumerate}
\end{proof}

\begin{lemma}\label{Yratio}
If $\displaystyle\mathbf{u},\,\mathbf{v}\in
\frac{1}{N}M_{1,\,2}(\mathbb{Z})\setminus
M_{1,\,2}(\mathbb{Z})$
such that $2\mathbf{v}\not\in M_{1,\,2}(\mathbb{Z})$, then
$\displaystyle\frac{Y_\mathbf{u}}{Y_\mathbf{v}}$ lies in $K_{(N)}$.
\end{lemma}
\begin{proof}
Observe first that $Y_\mathbf{v}\neq0$
by Lemma \ref{basicproperties} (iii).
If $2\mathbf{u}\in M_{1,\,2}(\mathbb{Z})$, then
the assertion is straightforward because
$Y_\mathbf{u}=0$ again by Lemma \ref{basicproperties} (iii).
Now, assume that $2\mathbf{u}\not\in M_{1,\,2}(\mathbb{Z})$.
We assert by Lemma \ref{pg} and
the definition (\ref{Yv}) that
\begin{equation*}
\frac{Y_\mathbf{u}}{Y_\mathbf{v}}=
g(\tau^{}_K)\quad\textrm{with}\quad
g=\frac{g_{2\mathbf{u}\,}
g_\mathbf{v}^4}
{g_\mathbf{u}^4\,g_{2\mathbf{v}}}.
\end{equation*}
If we set the family
$\{m(\mathbf{a})\}_{\mathbf{a}\in(1/N)
M_{1,\,2}(\mathbb{Z})\setminus M_{1,\,2}(\mathbb{Z})}$ of integers
so that
\begin{equation*}
\sum_{\mathbf{a}}
m(\mathbf{a})\cdot(\mathbf{a})=1\cdot(2\mathbf{u})+4\cdot
(\mathbf{v})+(-4)\cdot(\mathbf{u})+(-1)\cdot(2\mathbf{v})
\end{equation*}
by considering $(1/N)
M_{1,\,2}(\mathbb{Z})\setminus M_{1,\,2}(\mathbb{Z})$
as a free abelian group,
then we see that $\{m(\mathbf{a})\}_\mathbf{a}$ satisfy the three conditions
in Proposition \ref{modularity}.
Hence $g$ is a meromorphic modular function for $\Gamma(N)$
by Proposition \ref{modularity}. Moreover,
the infinite product expression of a Siegel function
given in (\ref{infiniteproduct}) shows that
$g$ belongs to $\mathcal{F}_N$.
Therefore we conclude by (\ref{generationray}) that $g(\tau^{}_K)$
belongs to $K_{(N)}$.
\end{proof}

For simplicity, we let
\begin{equation*}
T_N=\mathbb{Q}(E_K[N]).
\end{equation*}
Then we have the following description of the field $T_N$
in comparison with $K_{(N)}$.

\begin{theorem}\label{TN}
Assume that $K$ is different from $\mathbb{Q}(\sqrt{-1})$
and $\mathbb{Q}(\sqrt{-3})$.
\begin{enumerate}
\item[\textup{(i)}]
If $N=2$ and $d_K\equiv0\Mod{4}$, then
$T_N$ is the maximal real subfield of $K_{(2)}$
with $[K_{(2)}:T_N]=2$.
\item[\textup{(ii)}] If $N=2$ and $d_K\equiv1\Mod{4}$, then
$T_N=K_{(2)}$.
\item[\textup{(iii)}] If $N\geq3$, then
$T_N$ is the extension field of $K_{(N)}$
generated by $Y_{\left[\begin{smallmatrix}0&\frac{1}{N}\end{smallmatrix}\right]}$
with $[T_N:K_{(N)}]\leq2$.
\end{enumerate}
\end{theorem}
\begin{proof}
Note by Lemma \ref{basicproperties} (iii) that
\begin{equation}\label{QE2}
T_2=\mathbb{Q}\left(X_\mathbf{v}~|~
\mathbf{v}\in\frac{1}{2}M_{1,\,2}(\mathbb{Z})\setminus M_{1,\,2}(\mathbb{Z})\right).
\end{equation}
\begin{enumerate}
\item[(i)] If $N=2$ and $d_K\equiv0\Mod{4}$, then
(\ref{QE2}) and Lemma \ref{real} (i) yield that
$T_N$ is the maximal real subfield of $K_{(2)}$ with $[K_{(2)}:T_N]=2$.
\item[(ii)] If $N=2$ and $d_K\equiv1\Mod{4}$, then (\ref{QE2})
and Lemma \ref{real} (ii) show that
$T_N=K_{(2)}$.
\item[(iii)] If $N\geq3$, then we obtain
\begin{eqnarray*}
T_N&=&K_{(N)}\left(Y_\mathbf{v}~|~\mathbf{v}\in\frac{1}{N}M_{1,\,2}(\mathbb{Z})
\setminus M_{1,\,2}(\mathbb{Z})\right)\quad\textrm{by Lemma \ref{real} (ii)}\\
&=&K_{(N)}\left(
Y_{\left[\begin{smallmatrix}0&\frac{1}{N}\end{smallmatrix}\right]}
\right)\quad\textrm{by Lemma \ref{Yratio}}.
\end{eqnarray*}
It then follows from the relation
\begin{equation*}
Y_{\left[\begin{smallmatrix}
0 & \frac{1}{N}\end{smallmatrix}\right]}^2=
4X_{\left[\begin{smallmatrix}
0 & \frac{1}{N}\end{smallmatrix}\right]}^3-
\frac{J(J-1)}{27}
X_{\left[\begin{smallmatrix}
0 & \frac{1}{N}\end{smallmatrix}\right]}
-J\left(\frac{J-1}{27}\right)^2
\end{equation*}
and Proposition \ref{generatorX} (i) that
$[T_N:K_{(N)}]\leq 2$.
\end{enumerate}
\end{proof}

\begin{remark}\label{TQj}
Since $j(E_K)=j(\mathcal{O}_K)\in\mathbb{R}$ (\cite[p. 179]{Silverman94}),
we get by (\ref{generationHilbert}), (\ref{generationray})
and Theorem \ref{TN} that
\begin{equation*}
T_N\supseteq \mathbb{Q}(j(E_K)).
\end{equation*}
\end{remark}

\section {Representations attached to elliptic curves}

By using the Shimura reciprocity law we shall
determine the image of a $2$-dimensional representation of $\mathrm{Gal}(T_N/\mathbb{Q}(j(E_K)))$ attached to $E_K$.
\par
Recall that $T_N=\mathbb{Q}(E_K[N])$
is a finite extension of the real number field $\mathbb{Q}(j(E_K))$ by
Theorem \ref{TN} and Remark \ref{TQj}.
Since
$E_K$ is defined over $\mathbb{Q}(j(E_K))$,
we get $\sigma(T_N)=T_N$
for all $\sigma\in\mathrm{Gal}(\overline{\mathbb{Q}}/\mathbb{Q}(j(E_K)))$
(\cite[pp. 53--54]{Silverman09}),
which yields that $T_N$ is Galois over $\mathbb{Q}(j(E_K))$.
Define the right action of $\mathrm{Gal}(T_N/\mathbb{Q}(j(E_K)))$ on
the $\mathbb{Z}/N\mathbb{Z}$-module $E_K[N]$
as follows $\colon$
if $\sigma\in\mathrm{Gal}(T_N/\mathbb{Q}(j(E_K)))$, then
\begin{equation*}
[X_\mathbf{0}:Y_\mathbf{0}:0]^\sigma=
[X_\mathbf{0}:Y_\mathbf{0}:0]
\end{equation*}
and
\begin{equation*}
[X_\mathbf{v}:Y_\mathbf{v}:1]^\sigma=
[X_\mathbf{v}^\sigma:Y_\mathbf{v}^\sigma:1]
\quad\left(\mathbf{v}\in\frac{1}{N}M_{1,\,2}(\mathbb{Z})\setminus
M_{1,\,2}(\mathbb{Z})\right).
\end{equation*}
From this action, we achieve the faithful representation
\begin{equation*}
\rho^{}_N~:~\mathrm{Gal}(T_N/\mathbb{Q}(j(E_K)))~\rightarrow~\mathrm{GL}_2(\mathbb{Z}/N\mathbb{Z})~(\simeq
\mathrm{Aut}(E_K[N]))
\end{equation*}
with respect to the ordered basis
\begin{equation*}
\mathcal{B}=\left\{
\left[X_{\left[\begin{smallmatrix}
\frac{1}{N} & 0\end{smallmatrix}\right]}:
Y_{\left[\begin{smallmatrix}
\frac{1}{N} & 0\end{smallmatrix}\right]}:1
\right],~
\left[X_{\left[\begin{smallmatrix}
0& \frac{1}{N}\end{smallmatrix}\right]}:
Y_{\left[\begin{smallmatrix}
0 & \frac{1}{N}\end{smallmatrix}\right]}:1
\right]
\right\}
\end{equation*}
for $E_K[N]$.
We then have
\begin{equation}\label{XsX}
[X_\mathbf{v}:Y_\mathbf{v}:1]^\sigma
=[X_{\mathbf{v}\rho^{}_N(\sigma)}:Y_{\mathbf{v}\rho^{}_N(\sigma)}:1]
\quad\left(\mathbf{v}\in\frac{1}{N}M_{1,\,2}(\mathbb{Z})\setminus
M_{1,\,2}(\mathbb{Z})\right).
\end{equation}

\begin{lemma}\label{betaI}
Let $\beta\in M_2(\mathbb{Z})$ such that $\gcd(\det(\beta),\,N)=1$.
If
\begin{equation}\label{beta}
\mathbf{v}\beta\equiv\pm\mathbf{v}\Mod{M_{1,\,2}(\mathbb{Z})}\quad
\textrm{for each}~\mathbf{v}\in\frac{1}{N}M_{1,\,2}(\mathbb{Z})\setminus M_{1,\,2}(\mathbb{Z}),
\end{equation}
then $\beta\equiv\pm I_2\Mod{NM_2(\mathbb{Z})}$.
\end{lemma}
\begin{proof}
We deduce by putting
$\mathbf{v}=\begin{bmatrix}
\frac{1}{N} & 0\end{bmatrix}$
and $\begin{bmatrix}
0 & \frac{1}{N}\end{bmatrix}$ in (\ref{beta}) that
\begin{equation*}
\beta\equiv\pm\begin{bmatrix}
1 & \phantom{\pm}0\\
0 & \pm1
\end{bmatrix}\Mod{N M_2(\mathbb{Z})}.
\end{equation*}
Then we get the assertion by letting
$\mathbf{v}=\begin{bmatrix}\frac{1}{N} &
\frac{1}{N}\end{bmatrix}$
in (\ref{beta}).
\end{proof}

Now we denote by $\mathfrak{c}$ the complex conjugation on $\mathbb{C}$.

\begin{theorem}\label{compleximage}
We have
\begin{equation*}
\rho^{}_N(\mathfrak{c}|_{T_N})=\begin{bmatrix}1&\phantom{-}b_K\\0&-1\end{bmatrix}~
\textrm{or}~-\begin{bmatrix}1&\phantom{-}b_K\\0&-1\end{bmatrix}.
\end{equation*}
\end{theorem}
\begin{proof}
Let $\alpha=\rho^{}_N(\mathfrak{c}|_{T_N})$ ($\in\mathrm{GL}_2(\mathbb{Z}/N\mathbb{Z})$). We find
by (\ref{XsX}) and Lemma \ref{K-Lconjugation} that
\begin{equation*}
X_{\mathbf{v}\alpha}=\overline{X}_\mathbf{v}
=X_{\mathbf{v}\left[\begin{smallmatrix}1&\phantom{-}b_K\\0&-1\end{smallmatrix}\right]}
\quad\left(\mathbf{v}\in\frac{1}{N}M_{1,\,2}(\mathbb{Z})
\setminus M_{1,\,2}(\mathbb{Z})\right).
\end{equation*}
And we obtain by Lemma \ref{basicproperties} (i) that
\begin{equation*}
\mathbf{v}\alpha\equiv\pm\mathbf{v}\begin{bmatrix}1&\phantom{-}b_K\\
0&-1\end{bmatrix}
\Mod{M_{1,\,2}(\mathbb{Z})}\quad\textrm{for each}~\mathbf{v}\in\frac{1}{N}M_{1,\,2}(\mathbb{Z})
\setminus M_{1,\,2}(\mathbb{Z}).
\end{equation*}
Thus, it follows from Lemma \ref{betaI} that
\begin{equation*}
\alpha=
\rho^{}_N(\mathfrak{c}|_{T_N})=\begin{bmatrix}1&\phantom{-}b_K\\
0&-1\end{bmatrix}~\textrm{or}~
-\begin{bmatrix}1&\phantom{-}b_K\\
0&-1\end{bmatrix}.
\end{equation*}
\end{proof}

Let
\begin{equation*}
\widehat{W}_{K,\,N}=\left\langle W_{K,\,N},\,\begin{bmatrix}1&\phantom{-}b_K\\
0&-1\end{bmatrix}\right\rangle
\end{equation*}
as a subgroup of $\mathrm{GL}_2(\mathbb{Z}/N\mathbb{Z})$. We then see that
\begin{equation}\label{Windex}
[\widehat{W}_{K,\,N}:W_{K,\,N}]=\left\{
\begin{array}{ll}
1 & \textrm{if}~N=2~\textrm{and}~d_K\equiv0\Mod{4},\\
2 & \textrm{otherwise}.
\end{array}\right.
\end{equation}

\begin{theorem}\label{subgroup}
The image of $\rho^{}_N$ is a subgroup of $\widehat{W}_{K,\,N}$ of index
\begin{equation*}
[\widehat{W}_{K,\,N}:\rho^{}_N(\mathrm{Gal}(T_N/\mathbb{Q}(j(E_K))))]=
\left\{\begin{array}{ll}
1 & \textrm{if}~N=2,\\
1~\textrm{or}~2& \textrm{if}~N\geq3.
\end{array}\right.
\end{equation*}
\end{theorem}
\begin{proof}
Let $\sigma\in\mathrm{Gal}(T_N/\mathbb{Q}(j(E_K)))$. If we let
$\gamma=\rho^{}_N(\sigma)$, then we derive by (\ref{XsX}) that
\begin{equation*}
X_\mathbf{v}^\sigma=X_{\mathbf{v}\gamma}
\quad\left(\mathbf{v}\in\frac{1}{N}M_{1,\,2}(\mathbb{Z})
\setminus M_{1,\,2}(\mathbb{Z})\right).
\end{equation*}
On the other hand,
we achieve by Propositions \ref{reciprocity}, \ref{generatorX} (ii)
and Theorem \ref{compleximage} that
\begin{equation*}
X_\mathbf{v}^\sigma=X_{\mathbf{v}\alpha}\quad
\textrm{for some}~\alpha\in\widehat{W}_{K,\,N}
\quad\left(\mathbf{v}\in\frac{1}{N}M_{1,\,2}(\mathbb{Z})
\setminus M_{1,\,2}(\mathbb{Z})\right),
\end{equation*}
and hence
\begin{equation*}
X_{\mathbf{v}\gamma}=X_{\mathbf{v}\alpha}.
\end{equation*}
Therefore we attain by Lemma \ref{basicproperties} (i) that
\begin{equation*}
\mathbf{v}\gamma\equiv\pm\mathbf{v}\alpha\Mod{M_{1,\,2}(\mathbb{Z})}\quad
\textrm{for each}~\mathbf{v}\in\frac{1}{N}M_{1,\,2}(\mathbb{Z})\setminus M_{1,\,2}(\mathbb{Z}),
\end{equation*}
from which we have by Lemma \ref{betaI} that
\begin{equation*}
\gamma=\pm\alpha\quad\textrm{in}~\mathrm{GL}_2(\mathbb{Z}/N\mathbb{Z}).
\end{equation*}
This yields that
$\rho^{}_N(\mathrm{Gal}(T_N/\mathbb{Q}(j(E_K))))$ is a subgroup of $\widehat{W}_{K,\,N}$.
Furthermore, we find that
\begin{eqnarray*}
&&|\rho^{}_N(\mathrm{Gal}(T_N/\mathbb{Q}(j(E_K))))|\\&=&
|\mathrm{Gal}(T_N/\mathbb{Q}(j(E_K)))|\quad\textrm{because $\rho^{}_N$ is injective}\\
&=&\left\{\begin{array}{ll}
|\mathrm{Gal}(K_{(2)}/H_K)|& \textrm{if}~N=2~\textrm{and}~d_K\equiv0\Mod{4},\\
2|\mathrm{Gal}(K_{(2)}/H_K)| & \textrm{if}~N=2~\textrm{and}~d_K\equiv1\Mod{4},\\
2|\mathrm{Gal}(K_{(N)}/H_K)| ~\textrm{or}~
4|\mathrm{Gal}(K_{(N)}/H_K)|& \textrm{if}~N\geq3
\end{array}\right.\\
&&\hspace{10.5cm}\textrm{by (\ref{generationHilbert}) and Theorem \ref{TN}}\\
&=&
|W_{K,\,N}/\{\pm I_2\}|\times
\left\{\begin{array}{ll}
1& \textrm{if}~N=2~\textrm{and}~d_K\equiv0\Mod{4},\\
2 & \textrm{if}~N=2~\textrm{and}~d_K\equiv1\Mod{4},\\
2~\textrm{or}~4& \textrm{if}~N\geq3
\end{array}\right.\\
&&\hspace{10.5cm}\textrm{by Proposition \ref{reciprocity}}\\
&=&
|W_{K,\,N}|\times
\left\{\begin{array}{ll}
1& \textrm{if}~N=2~\textrm{and}~d_K\equiv0\Mod{4},\\
2 & \textrm{if}~N=2~\textrm{and}~d_K\equiv1\Mod{4},\\
1~\textrm{or}~2& \textrm{if}~N\geq3\\
\end{array}\right.\\
&=&|\widehat{W}_{K,\,N}|\times
\left\{\begin{array}{ll}
1& \textrm{if}~N=2,\\
\frac{1}{2}~\textrm{or}~1& \textrm{if}~N\geq3
\end{array}\right.
\quad\textrm{by (\ref{Windex})}.
\end{eqnarray*}
This completes the proof.
\end{proof}

\begin{remark}\label{not-I2}
Consider the case where $[\widehat{W}_{K,\,N}:\rho^{}_N(\mathrm{Gal}(T_N/\mathbb{Q}(j(E_K))))]=2$, equivalently,
\begin{equation*}
N\geq3\quad\textrm{and}\quad Y_{\left[\begin{smallmatrix}0&\frac{1}{N}
\end{smallmatrix}\right]}\in K_{(N)}.
\end{equation*}
If $\rho^{}_N(\sigma)=-I_2$ for some $\sigma\in
\mathrm{Gal}(T_N/\mathbb{Q}(j(E_K)))$, then we get
by (\ref{XsX}),
Lemma \ref{basicproperties} (i) and (ii) that
\begin{equation*}
X_{\mathbf{v}}^\sigma=
X_{\mathbf{v}(-I_2)}=
X_{\mathbf{v}}
\quad\textrm{for all}~\mathbf{v}\in\frac{1}{N}
M_{1,\,2}(\mathbb{Z})\setminus M_{1,\,2}(\mathbb{Z})
\end{equation*}
and
\begin{equation*}
Y_{\left[\begin{smallmatrix}
0&\frac{1}{N}
\end{smallmatrix}\right]}^\sigma=
Y_{\left[\begin{smallmatrix}
0&\frac{1}{N}
\end{smallmatrix}\right](-I_2)}=
-Y_{\left[\begin{smallmatrix}
0&\frac{1}{N}
\end{smallmatrix}\right]}.
\end{equation*}
This implies by Lemma \ref{real} (ii)
that $Y_{\left[\begin{smallmatrix}
0&\frac{1}{N}
\end{smallmatrix}\right]}$
does not belong to $K_{(N)}$,
which gives a contradiction.
Therefore we must have
\begin{equation*}
\rho^{}_N(\mathrm{Gal}(T_N/\mathbb{Q}(j(E_K))))\not\ni -I_2\quad
\textrm{and}\quad
\langle
\rho^{}_N(\mathrm{Gal}(T_N/\mathbb{Q}(j(E_K)))),\,-I_2\rangle
=\widehat{W}_{K,\,N}.
\end{equation*}
\end{remark}

\section {Some $p$-adic Galois representations}\label{Sect8}

Let $p$ be a prime. We shall investigate the image of a $p$-adic Galois representation of $\mathrm{Gal}(\overline{\mathbb{Q}}/\mathbb{Q}(j(E_K)))$
induced from its right action on the $p$-adic Tate module of the elliptic curve $E_K$ (Theorem \ref{pre2}).

\begin{lemma}\label{directproduct}
There is a positive integer $k$ such that
\begin{equation*}
\rho^{}_{p^n}(\mathrm{Gal}(T_{p^n}/\mathbb{Q}(j(E_K))))\cap\{-I_2\}=
\left\{\begin{array}{cl}
\{-I_2\} & \vspace{-0.2cm}\textrm{for all}~n\geq k\\
&\vspace{-0.2cm}~~~~~~~~~~~~~~~~~\textrm{or}\\
\emptyset & \textrm{for all}~n\geq k.
\end{array}\right.
\end{equation*}
\end{lemma}
\begin{proof}
If there is a positive integer $k$ such that
$\rho^{}_{p^n}(\mathrm{Gal}(T_{p^n}/\mathbb{Q}(j(E_K))))\cap\{-I_2\}=\{-I_2\}$
for all $n\geq k$, then
we are done.
\par
Now, consider the case where
$\rho^{}_{p^k}(\mathrm{Gal}(T_{p^k}/\mathbb{Q}(j(E_K))))\cap\{-I_2\}=\emptyset$
for some positive integer $k$. If
$\rho^{}_{p^\ell}(\mathrm{Gal}(T_{p^\ell}/\mathbb{Q}(j(E_K))))$ contains $-I_2$
for some $\ell>k$, then
the commutative diagram in Figure \ref{commutative}
asserts that
$\rho^{}_{p^k}(\mathrm{Gal}(T_{p^k}/\mathbb{Q}(j(E_K))))$ also contains $-I_2$,
which gives rise to a contradiction.
Therefore we conclude
\begin{equation*}
\rho^{}_{p^n}(\mathrm{Gal}(T_{p^n}/\mathbb{Q}(j(E_K))))\cap\{-I_2\}=\emptyset\quad
\textrm{for all}~n\geq k.
\end{equation*}

\begin{figure}[h]
\begin{equation*}
\xymatrixcolsep{2pc}\xymatrix{
\mathrm{Gal}(T_{p^\ell}/\mathbb{Q}(j(E_K))) \ar[r]^{~~~~\rho^{}_{p^\ell}}
\ar[d]^{}_{\textrm{restriction}} & \mathrm{GL}_2(\mathbb{Z}/p^\ell\mathbb{Z})
\ar[d]^{\textrm{reduction}}\\
\mathrm{Gal}(T_{p^k}/\mathbb{Q}(j(E_K))) \ar[r]^{~~~~\rho^{}_{p^k}} & \mathrm{GL}_2(\mathbb{Z}/p^k\mathbb{Z})
}
\end{equation*}
\caption{A commutative diagram of representations}\label{commutative}
\end{figure}
\end{proof}

\begin{lemma}\label{-I2}
Let $p$ be an odd prime for which it does not split in $K$ and
$[K_{(p)}:H_K]$ is even.
Then, $\rho^{}_{p^n}(\mathrm{Gal}(T_{p^n}/\mathbb{Q}(j(E_K))))$ contains
$-I_2$ for every $n\geq1$.
\end{lemma}
\begin{proof}
First, note by Theorem \ref{TN} (iii) that $T_{p^n}$ contains $K_{(p^n)}$.
Since $\rho^{}_{p^n}$ is injective and
\begin{equation*}
|\mathrm{Gal}(T_{p^n}/H_K)|=
[T_{p^n}:K_{(p)}][K_{(p)}:H_K]
\end{equation*}
is even by hypothesis, the subgroup $\rho^{}_{p^n}(\mathrm{Gal}(T_{p^n}/H_K))$
of $\rho^{}_{p^n}(\mathrm{Gal}(T_{p^n}/\mathbb{Q}(j(E_K))))$
contains an element
$\gamma=\begin{bmatrix}a&b\\c&d\end{bmatrix}\in
\mathrm{GL}_2(\mathbb{Z}/p^n\mathbb{Z})$ of order $2$.
Here we observe that in $M_2(\mathbb{Z}/p^n\mathbb{Z})$,
$\gamma$ satisfies
\begin{equation*}
\gamma^2-(a+b)\gamma+(ad-bc)I_2=O_2,\quad
\gamma\neq I_2\quad\textrm{and}\quad\gamma^2=I_2,
\end{equation*}
and so
\begin{equation}\label{a+d}
(a+d)\gamma=(1+ad-bc)I_2.
\end{equation}
Now, there are two possibilities $\colon$ $a+d\not\equiv
0\Mod{p}$ or $a+d\equiv0\Mod{p}$.
\begin{enumerate}
\item[Case 1.] If $a+d\not\equiv0\Mod{p}$, then one can readily get
$\gamma=-I_2$.
\item[Case 2.] If $a+d\equiv0\Mod{p}$, then by (\ref{a+d}) we attain
\begin{equation}\label{a+d'}
1-d^2-bc\equiv0\Mod{p}.
\end{equation}
On the other hand, since
$\rho^{}_{p^n}(\mathrm{Gal}(T_{p^n}/H_K))$
is a subgroup of $W_{K,\,p^n}$ by Proposition \ref{reciprocity}
and the proof of Theorem \ref{subgroup}, we have
\begin{equation*}
\gamma=\begin{bmatrix}
t-b_Ks & -c_Ks\\
s & t
\end{bmatrix}
\quad\textrm{for some}~s,\,t\in\mathbb{Z}/p^n\mathbb{Z}.
\end{equation*}
Thus we establish by the assumption $a+d\equiv0\Mod{p}$ and (\ref{a+d'}) that
\begin{equation*}
(t-b_Ks)+t\equiv0\Mod{p}\quad
\textrm{and}\quad1-t^2+c_Ks^2\equiv0\Mod{p}.
\end{equation*}
It then follows from the fact $b_K^2-4c_K=d_K$ that
\begin{equation*}
d_Ks^2\equiv4\Mod{p},
\end{equation*}
and hence
\begin{equation*}
p\nmid d_K\quad\textrm{and}\quad\left(\frac{d_K}{p}\right)=1.
\end{equation*}
However, this case does not happen by the hypothesis on $p$.
\end{enumerate}
\end{proof}

\begin{remark}
For an odd prime $p$, we have the degree formula
\begin{equation*}
[K_{(p)}:H_K]=\left\{
\begin{array}{ll}
\displaystyle\frac{(p-1)^2}{2} & \textrm{if $p$ splits in $K$},\\
\displaystyle\frac{(p-1)p}{2} & \textrm{if $p$ is ramified in $K$},\\
\displaystyle\frac{p^2-1}{2} & \textrm{if $p$ is inert in $K$}
\end{array}\right.
\end{equation*}
(\cite[Theorem 1 in Chapter VI]{Lang94}).
So, $p$ satisfies the hypothesis of Lemma \ref{-I2} if and only if
\begin{equation*}
\left\{\begin{array}{l}
p\,|\,d_K\\
p\equiv1\Mod{4},
\end{array}\right.\quad\textrm{or}\quad
\left(\frac{d_K}{p}\right)=-1.
\end{equation*}
\end{remark}

\begin{theorem}\label{main1}
Assume that $K$ is different from $\mathbb{Q}(\sqrt{-1})$ and
$\mathbb{Q}(\sqrt{-3})$. Let
\begin{equation*}
\widehat{W}_{K,\,p^\infty}=\left\langle
W_{K,\,p^\infty},\,
\begin{bmatrix}1&\phantom{-}b_K\\0&-1\end{bmatrix}
\right\rangle\quad(\subseteq\mathrm{GL}_2(\mathbb{Z}_p)).
\end{equation*}
Then, there exists a $p$-adic Galois representation
$\rho^{}_{p^\infty}:\mathrm{Gal}(\overline{\mathbb{Q}}/\mathbb{Q}(j(E_K)))\rightarrow\mathrm{GL}_2(\mathbb{Q}_p)$
whose image is a subgroup of $\widehat{W}_{K,\,p^\infty}$ of index $1$ or $2$.
In particular, if $p$ is an odd prime
for which it does not split in $K$ and
$[K_{(p)}:H_K]$ is even, then
the image of $\rho^{}_{p^\infty}$ coincides with
$\widehat{W}_{K,\,p^\infty}$.
\end{theorem}
\begin{proof}
For each $n\geq1$, we get by Theorem \ref{subgroup} and Remark \ref{not-I2} that
\begin{equation}\label{Wr}
\widehat{W}_{K,\,p^n}=\langle\rho^{}_{p^n}
(\mathrm{Gal}(T_{p^n}/\mathbb{Q}(j(E_K)))),\,-I_2\rangle.
\end{equation}
Moreover, by Lemma \ref{directproduct} there is a positive integer $k$ such that
\begin{equation*}
\widehat{W}_{K,\,p^n}=\left\{\begin{array}{l}
\rho^{}_{p^n}(\mathrm{Gal}(T_{p^n}/\mathbb{Q}(j(E_K)))),~\textrm{or}\\
\rho^{}_{p^n}(\mathrm{Gal}(T_{p^n}/\mathbb{Q}(j(E_K))))\oplus\langle -I_2\rangle
\end{array}\right.\quad\textrm{for all}~n\geq k.
\end{equation*}
And we obtain by taking the inverse limit on both sides of (\ref{Wr}) that
\begin{equation*}
\widehat{W}_{K,\,p^\infty}=\left\langle
\varprojlim_{n\geq1}\rho^{}_{p^n}(\mathrm{Gal}(T_{p^n}/\mathbb{Q}(j(E_K)))),\,-I_2\right\rangle
\quad\textrm{in}~\mathrm{GL}_2(\mathbb{Z}_p).
\end{equation*}
Let $\displaystyle T_{p^\infty}=\bigcup_{n\geq1}T_{p^n}$.
By composing the homomorphisms in Figure \ref{compositiondiagram}
we attain the $p$-adic Galois representation
\begin{equation*}
\rho^{}_{p^\infty}~:~\mathrm{Gal}(\overline{\mathbb{Q}}/\mathbb{Q}(j(E_K)))~\rightarrow~
\mathrm{GL}_2(\mathbb{Q}_p)
\end{equation*}
which satisfies
\begin{equation*}
[\widehat{W}_{K,\,p^\infty}:\rho^{}_{p^\infty}
(\mathrm{Gal}(\overline{\mathbb{Q}}/\mathbb{Q}(j(E_K))))]=1~\textrm{or}~2.
\end{equation*}
This proves the first part of the theorem.

\begin{figure}[h]
\begin{equation*}
\xymatrixcolsep{2pc}\xymatrix{
~~~~\mathrm{Gal}(\overline{\mathbb{Q}}/\mathbb{Q}(j(E_K)))~~~~ \ar[r]^{\textrm{restriction}} &
~~\mathrm{Gal}(T_{p^\infty}/\mathbb{Q}(j(E_K)))~~ \ar[r]^\sim  &
~~\displaystyle\varprojlim_{n\geq1}\mathrm{Gal}(T_{p^n}/\mathbb{Q}(j(E_K)))~~ \ar[d]^\wr\\
~~\mathrm{GL}_2(\mathbb{Q}_p)~~ &
~~\widehat{W}_{K,\,p^\infty}~~ \ar@{_{(}->}[l]&
~~\displaystyle\varprojlim_{n\geq1}\rho^{}_{p^n}(\mathrm{Gal}(T_{p^n}/\mathbb{Q}(j(E_K))))~~ \ar@{_{(}->}[l]
}
\end{equation*}
\caption{A sequence of homomorphisms}\label{compositiondiagram}
\end{figure}
\par
In particular, let
$p$ be an odd prime
for which it does not split in $K$ and
$[K_{(p)}:H_K]$ is even.
Since $\rho^{}_{p^n}(\mathrm{Gal}(T_{p^n}/\mathbb{Q}(j(E_K))))$ contains $-I_2$ for all $n\geq1$ by
Lemma \ref{-I2}, we achieve by (\ref{Wr}) that
\begin{equation*}
\rho^{}_{p^\infty}(\mathrm{Gal}(\overline{\mathbb{Q}}/\mathbb{Q}(j(E_K))))=
\varprojlim_{n\geq1}\rho^{}_{p^n}(\mathrm{Gal}(T_{p^n}/\mathbb{Q}(j(E_K))))=
\widehat{W}_{K,\,p^\infty}.
\end{equation*}
\end{proof}

\section {Extended form class groups}\label{EFCG}

From now on, we let $K$ be an arbitrary imaginary quadratic field
including both $\mathbb{Q}(\sqrt{-1})$ and $\mathbb{Q}(\sqrt{-3})$, and let $N$ be a positive integer.
We shall introduce the recent work of Eum, Koo and Shin
(\cite{E-K-S}) on constructing a form class group
isomorphic to $\mathrm{Gal}(K_{(N)}/K)$.
\par
Let
$\mathcal{Q}_N(d_K)$ be the set of certain binary quadratic forms over $\mathbb{Z}$ given by
\begin{equation*}
\mathcal{Q}_N(d_K)=\left\{ax^2+bxy+cy^2\in\mathbb{Z}[x,\,y]
~|~\gcd(a,\,b,\,c)=1,~b^2-4ac=d_K,~a>0,~\gcd(a,\,N)=1\right\}.
\end{equation*}
The congruence subgroup
\begin{equation*}
\Gamma_1(N)=\left\{\gamma\in\mathrm{SL}_2(\mathbb{Z})~|~
\gamma\equiv\begin{bmatrix}
1&\mathrm{*}\\0&1
\end{bmatrix}\Mod{NM_2(\mathbb{Z})}\right\}
\end{equation*}
acts on the set $\mathcal{Q}_N(d_K)$
from the right, which induces the equivalence relation $\sim^{}_N$ as
\begin{equation*}
Q\sim^{}_N Q'\quad\Longleftrightarrow\quad
Q'=Q^\gamma=Q\left(\gamma\begin{bmatrix}x\\y\end{bmatrix}\right)~
\textrm{for some}~\gamma\in\Gamma_1(N)
\end{equation*}
(\cite[Proposition 2.1 and Definition 2.2]{E-K-S}).
Denote the set of equivalence classes by $\mathcal{C}_N(d_K)$, namely,
\begin{equation*}
\mathcal{C}_N(d_K)=\mathcal{Q}_N(d_K)/\sim^{}_N
=\{[Q]^{}_N~|~Q\in\mathcal{Q}_N(d_K)\}.
\end{equation*}
For each $Q=ax^2+bxy+cy^2\in\mathcal{Q}_N(d_K)$, let $\omega_Q$ be the zero
of the quadratic polynomial $Q(x,\,1)$ lying in $\mathbb{H}$, that is,
\begin{equation*}
\omega_Q=\frac{-b+\sqrt{d_K}}{2a}.
\end{equation*}
Then one can readily check that for $Q,\,Q'\in\mathcal{Q}_N(d_K)$
and $\alpha\in\mathrm{SL}_2(\mathbb{Z})$
\begin{equation}\label{wQequal}
\omega_Q=\omega^{}_{Q'}\quad\Longleftrightarrow\quad
Q=Q'
\end{equation}
and
\begin{equation}\label{wQgamma}
\omega^{}_{Q^\alpha}=\alpha^{-1}(\omega_Q).
\end{equation}
Let $\mathrm{Cl}(N\mathcal{O}_K)$ be the ray class group of $K$ modulo $N\mathcal{O}_K$,
i.e.,
\begin{equation*}
\mathrm{Cl}(N\mathcal{O}_K)=I_K(N\mathcal{O}_K)/P_{K,\,1}(N\mathcal{O}_K)
\end{equation*}
where $I_K(N\mathcal{O}_K)$ is the group of all
fractional ideals of $K$ relatively prime to $N\mathcal{O}_K$
and $P_{K,\,1}(N\mathcal{O}_K)$ is its subgroup given by
\begin{equation*}
P_{K,\,1}(N\mathcal{O}_K)=\left\{
\nu\mathcal{O}_K~|~\nu\in K^\times~\textrm{satisfies}~\nu\equiv^*1\Mod{N\mathcal{O}_K}\right\}.
\end{equation*}
Here, $\nu\equiv^* 1 \Mod{N\mathcal{O}_K}$
means that $\mathrm{ord}_\mathfrak{p}(\nu-1)\geq \mathrm{ord}_\mathfrak{p}(N\mathcal{O}_K)$ for all prime ideals $\mathfrak{p}$ of $\mathcal{O}_K$ dividing $N\mathcal{O}_K$.
Eum et al. showed that
the map
\begin{equation*}
\begin{array}{ccl}
\mathcal{C}_N(d_K)&\rightarrow&\mathrm{Cl}(N\mathcal{O}_K)\\
\mathrm{[}Q]&\mapsto&[[\omega_Q,\,1]]=[\mathbb{Z}\omega_Q+\mathbb{Z}]
\end{array}
\end{equation*}
is a well-defined bijection
(\cite[Theorem 2.9]{E-K-S}).
Therefore the set $\mathcal{C}_N(d_K)$ can be regarded as a group
isomorphic to $\mathrm{Cl}(N\mathcal{O}_K)$
through the above bijection,
called the \textit{extended form class group} of discriminant $d_K$ and level $N$.
The principal form
\begin{equation*}
Q_0=\left\{
\begin{array}{ll}
\displaystyle x^2+xy+\frac{1-d_K}{4}y^2 & \textrm{if}~d_K\equiv1\Mod{4},\\
\displaystyle x^2-\frac{d_K}{4}y^2 & \textrm{if}~d_K\equiv0\Mod{4}
\end{array}\right.
\end{equation*}
becomes the identity element of $\mathcal{C}_N(d_K)$.

\begin{definition}\label{+invariant}
Let $Q=ax^2+bxy+cy^2\in\mathcal{Q}_N(d_K)$
and $f\in\mathcal{F}_N$.
We define
\begin{equation*}
f([Q]^{}_N)=f^{\gamma_Q^{-1}}(-\overline{\omega}^{}_Q)
\end{equation*}
where
\begin{equation*}
\gamma^{}_Q=\begin{bmatrix}
1& \frac{b+b_K}{2}\\0&a\end{bmatrix}
\quad(\in\mathrm{GL}_2(\mathbb{Z}/N\mathbb{Z})/\{\pm I_2\}
\simeq\mathrm{Gal}(\mathcal{F}_N/\mathcal{F}_1)).
\end{equation*}
This value depends only on the class $[Q]^{}_N$
in $\mathcal{C}_N(d_K)$ if it is finite (\cite[Proposition 3.3 and Lemma 3.7]{Yoon}).
In particular, one can derive the equality $f([Q_0]^{}_N)=f(\tau^{}_K)$.
\end{definition}

The fact that $\mathrm{Cl}(N\mathcal{O}_K)\simeq
\mathrm{Gal}(K_{(N)}/K)$ and the original version of the Shimura reciprocity law
(\cite[Theorem 6.31]{Shimura})
lead to the following proposition.

\begin{proposition}\label{extended}
We have the isomorphism
\begin{equation*}
\begin{array}{ccl}
\phi^{}_N~:~\mathcal{C}_N(d_K)&\stackrel{\sim}{\rightarrow}&
\mathrm{Gal}(K_{(N)}/K)\vspace{0.1cm}\\
~~~~~~~~~\mathrm{[}Q\mathrm{]}_N&\mapsto&
\left(
f([Q_0]^{}_N)=f(\tau^{}_K)\mapsto f([Q]^{}_N)~|~f\in\mathcal{F}_N~\textrm{is finite at}~\tau^{}_K\right).
\end{array}
\end{equation*}
\end{proposition}
\begin{proof}
See \cite[Theorem 3.10]{E-K-S} or
\cite[Theorem 3.8]{Yoon}.
\end{proof}

\begin{remark}\label{extendedQQ}
Let $Q,\,Q'\in\mathcal{Q}_N(d_K)$
and $f\in\mathcal{F}_N$ such that $f([Q]^{}_N)$ is finite.
By the homomorphism property of $\phi^{}_N$, we obtain
\begin{equation*}
\phi^{}_N([Q']^{}_N)\left(f([Q]^{}_N)\right)=
f([Q']^{}_N[Q]^{}_N).
\end{equation*}
\end{remark}

\section {Definite form class groups}

By improving the results introduced in $\S$\ref{EFCG}
we shall construct
the definite form class group which is isomorphic to $\mathrm{Gal}(K_{(N)}/\mathbb{Q})$,
and define the definite form class invariants (Theorems \ref{pre3} and \ref{pre4}).
First, we need the following lemma.

\begin{lemma}\label{semi}
The extension $K_{(N)}/\mathbb{Q}$ is Galois
and
\begin{equation*}
\mathrm{Gal}(K_{(N)}/\mathbb{Q})=
\mathrm{Gal}(K_{(N)}/K)\rtimes
\langle\mathfrak{c}|_{K_{(N)}}\rangle.
\end{equation*}
\end{lemma}
\begin{proof}
Since
$\mathfrak{c}|_K\neq\mathrm{id}_K$ and $[K:\mathbb{Q}]=2$, we see that
\begin{equation*}
\rho,\,\mathfrak{c}\rho\quad(\rho\in\mathrm{Gal}(K_{(N)}/K))
\end{equation*}
are all the distinct embeddings of $K_{(N)}$ into $\mathbb{C}$.
To prove the first part of the lemma, it suffices to show that
$\mathfrak{c}(K_{(N)})=K_{(N)}$.
We find that
\begin{eqnarray*}
&&\mathrm{Gal}(\mathfrak{c}(K_{(N)})/\mathfrak{c}(K))\\
&=&\mathfrak{c}\,\mathrm{Gal}(K_{(N)}/K)\,\mathfrak{c}^{-1}\\
&\simeq&I_{\mathfrak{c}(K)}(\mathfrak{c}(N\mathcal{O}_K))
/\{\mathfrak{c}(\mathfrak{a})~|~\mathfrak{a}\in P_{K,\,1}(N\mathcal{O}_K)\}\\
&&\textrm{since the Artin symbol satisfies}~
\left(\frac{\mathfrak{c}(K_{(N)})/\mathfrak{c}(K)}{\mathfrak{c}(\mathfrak{a})}\right)
=\mathfrak{c}\left(\frac{K_{(N)}/K}{\mathfrak{a}}\right)\mathfrak{c}^{-1}\\
&=&I_K(N\mathcal{O}_K)/P_{K,\,1}(N\mathcal{O}_K)\\
&&\textrm{because if $\alpha\in K^\times$ satisfies}~\alpha\equiv^*1\Mod{N\mathcal{O}_K},~
\textrm{then}~\mathfrak{c}(\alpha)\equiv^*1\Mod{N\mathcal{O}_K}.
\end{eqnarray*}
Now, the existence theorem of class field theory
(\cite[Theorem 8.6]{Cox} or \cite[$\S$V.9]{Janusz}) implies $\mathfrak{c}(K_{(N)})=K_{(N)}$.
\par
Observe that $\mathrm{Gal}(K_{(N)}/K)$ is normal in $\mathrm{Gal}(K_{(N)}/\mathbb{Q})$
because
\begin{equation*}
[\mathrm{Gal}(K_{(N)}/\mathbb{Q}):\mathrm{Gal}(K_{(N)}/K)]=[K:\mathbb{Q}]=2.
\end{equation*}
Furthermore, we see that
$\mathrm{Gal}(K_{(N)}/K)\cap\langle\mathfrak{c}|_{K_{(N)}}\rangle=
\{\mathrm{id}_{K_{(N)}}\}$,
and so
\begin{equation*}
\mathrm{Gal}(K_{(N)}/\mathbb{Q})=\mathrm{Gal}(K_{(N)}/K)
\langle\mathfrak{c}|_{K_{(N)}}\rangle.
\end{equation*}
Therefore, $\mathrm{Gal}(K_{(N)}/\mathbb{Q})$
is the semidirect  product of $\mathrm{Gal}(K_{(N)}/K)$ and
$\langle\mathfrak{c}|_{K_{(N)}}\rangle$.
\end{proof}

Let us denote by
\begin{eqnarray*}
\mathcal{Q}_N^\pm(d_K)&=&\{ax^2+bxy+cy^2\in\mathbb{Z}[x,\,y]~|~
\gcd(a,\,b,\,c)=1,~b^2-4ac=d_K,~\gcd(a,\,N)=1\},\\
\mathcal{Q}_N^+(d_K)&=&\{Q\in\mathcal{Q}_N^\pm(d_K)~|~
\textrm{$Q$ is positive definite}\}~=~\mathcal{Q}_N(d_K),\\
\mathcal{Q}_N^-(d_K)&=&\{Q\in\mathcal{Q}_N^\pm(d_K)~|~
\textrm{$Q$ is negative definite}\}.
\end{eqnarray*}
We then see that
\begin{eqnarray}
\mathcal{Q}_N^\pm(d_K)&=&\mathcal{Q}_N^+(d_K)
\cupdot\mathcal{Q}_N^-(d_K),\label{union}\\
\mathcal{Q}_N^-(d_K)&=&\{-Q=(-1)Q~|~
Q\in\mathcal{Q}_N^+(d_K)\}.\label{-+}
\end{eqnarray}
The action of $\Gamma_1(N)$ on $\mathcal{Q}_N(d_K)$
can be naturally extended to that on $\mathcal{Q}_N^\pm(d_K)$ as
\begin{equation*}
Q^\gamma=Q\left(\gamma\begin{bmatrix}x\\y\end{bmatrix}\right)
\quad(Q\in\mathcal{Q}_N^\pm(d_K),\,\gamma\in\Gamma_1(N)),
\end{equation*}
which induces the equivalence relation $\sim^\pm_N$ on $\mathcal{Q}_N^\pm(d_K)$.
Denote by
\begin{equation*}
\mathcal{C}_N^\pm(d_K)~=~\mathcal{Q}_N^\pm(d_K)/\sim^\pm_N~=~
\{[Q]^\pm_N~|~Q\in\mathcal{Q}^\pm_N(d_K)\}
\end{equation*}
the set of equivalence classes.
For $Q\in\mathcal{Q}_N^\pm(d_K)$, we let
\begin{equation*}
\mathrm{sgn}(Q)=\left\{\begin{array}{rl}
1 & \textrm{if}~Q\in\mathcal{Q}_N^+(d_K),\\
-1 & \textrm{if}~Q\in\mathcal{Q}_N^-(d_K).
\end{array}\right.
\end{equation*}
Observe that if $\mathrm{sgn}(Q)=1$,
then $[Q]^{}_N$ in $\mathcal{C}_N(d_K)$
coincides with $[Q]^\pm_N$ in $\mathcal{C}^\pm_N(d_K)$.
Thus, one can simply write without confusion $\sim^{}_N$ and $[Q]^{}_N$ for
$\sim^\pm_N$ and $[Q]^\pm_N$, respectively. Recall from
Proposition \ref{extended} that
we have the explicit isomorphism $\phi^{}_N:\mathcal{C}_N(d_K)\rightarrow
\mathrm{Gal}(K_{(N)}/K)$.

\begin{theorem}\label{C+-}
The set $\mathcal{C}_N^\pm(d_K)$ can be given
a group structure isomorphic to $\mathrm{Gal}(K_{(N)}/\mathbb{Q})$ so that
it contains $\mathcal{C}_N(d_K)$ as a subgroup
and the element $[-Q_0]^{}_N$ corresponds to
$\mathfrak{c}|_{K_{(N)}}$.
\end{theorem}
\begin{proof}
Note that if $Q\sim^{}_N Q'$ for $Q,\,Q'\in\mathcal{Q}_N^\pm(d_K)$, then $\mathrm{sgn}(Q)=\mathrm{sgn}(Q')$.
This observation, together with (\ref{union}) and (\ref{-+}), implies that
\begin{eqnarray*}
\mathcal{C}_N^\pm(d_K)&\rightarrow&\mathcal{C}_N(d_K)\times\{1,\,-1\}\\
\mathrm{[}Q\mathrm{]}_N & \mapsto & ([\mathrm{sgn}(Q)Q]^{}_N,\,\mathrm{sgn}(Q))
\end{eqnarray*}
is a well-defined bijection.
And, by Lemma \ref{semi} we deduce the bijection
\begin{equation}\label{phi+-}
\begin{array}{ccl}
\mathcal{C}_N^\pm(d_K)&\rightarrow&
\mathrm{Gal}(K_{(N)}/\mathbb{Q})=\mathrm{Gal}(K_{(N)}/K)
\rtimes\langle\mathfrak{c}|_{K_{(N)}}\rangle\vspace{0.1cm}\\
\mathrm{[}Q\mathrm{]}_N & \mapsto &
\phi^{}_N([\mathrm{sgn}(Q)Q]^{}_N)\,
\mathfrak{c}|_{K_{(N)}}^{\frac{1-\mathrm{sgn}(Q)}{2}}.
\end{array}
\end{equation}
This proves the theorem.
\end{proof}

\begin{remark}\label{sgnproduct}
\begin{enumerate}
\item[(i)] Since $\mathfrak{c}|_{K_{(N)}}$ is of order $2$, we get
$[-Q_0]^{}_N[-Q_0]^{}_N=[Q_0]^{}_N$.
\item[(ii)] Let $Q,\,Q'\in\mathrm{Q}_N^\pm(d_K)$.
Since $\mathrm{sgn}(Q)$ depends only on the class
$[Q]^{}_N$, we may also write $\mathrm{sgn}([Q]^{}_N)$ for $\mathrm{sgn}(Q)$.
And we see from the isomorphism in (\ref{phi+-}) (and the concept
of a semidirect product) that
\begin{equation*}
\mathrm{sgn}([Q]^{}_N[Q']^{}_N)=\mathrm{sgn}([Q]^{}_N)\mathrm{sgn}([Q']^{}_N).
\end{equation*}
\end{enumerate}
\end{remark}

Now, we let
\begin{equation*}
\phi^\pm_N~:~\mathcal{C}^\pm_N(d_K)\rightarrow\mathrm{Gal}(K_{(N)}/\mathbb{Q})
\end{equation*}
be the isomorphism stated in (\ref{phi+-}).
By extending Definition \ref{+invariant} we define the
definite form class invariants as follows.

\begin{definition}\label{definvariant}
Let $Q\in\mathcal{Q}_N^\pm(d_K)$ and $f\in\mathcal{F}_N$.
Define
\begin{eqnarray*}
f([Q]^{}_N)&=&
\phi^\pm_N([-Q_0]_N^{\frac{1-\mathrm{sgn}([Q]^{}_N)}{2}})
(f([-Q_0]_N^{\frac{1-\mathrm{sgn}([Q]^{}_N)}{2}}[Q]^{}_N))\\
&=&
\left\{\begin{array}{cl}
f([Q]^{}_N) & \textrm{if}~\mathrm{sgn}([Q]^{}_N)=1,\vspace{0.1cm}\\
\overline{f([-Q_0]^{}_N[Q]^{}_N)} & \textrm{if}~\mathrm{sgn}([Q]^{}_N)=-1.
\end{array}\right.
\end{eqnarray*}
\end{definition}

\begin{remark}\label{sgnsgn}
Observe that
\begin{equation*}
\mathrm{sgn}
([-Q_0]_N^\frac{1-\mathrm{sgn}([Q]^{}_N)}{2})=\mathrm{sgn}([Q]^{}_N)
\end{equation*}
and so
\begin{equation*}
\mathrm{sgn}([-Q_0]_N^\frac{1-\mathrm{sgn}([Q]^{}_N)}{2}[Q]^{}_N)=1
\end{equation*}
by Remark \ref{sgnproduct}.
\end{remark}

\begin{theorem}\label{pmQQ}
Let $Q,\,Q'\in\mathcal{Q}^\pm_N(d_K)$ and $f\in\mathcal{F}_N$.
If $f([Q]^{}_N)$ is finite, then
\begin{equation*}
\phi^\pm_N([Q']^{}_N)(f([Q]^{}_N))=
f([Q']^{}_N[Q]^{}_N).
\end{equation*}
\end{theorem}
\begin{proof}
We derive that
\begin{eqnarray*}
&&\phi^\pm_N([Q']^{}_N)(f([Q]^{}_N))\\&=&
\phi^\pm_N([Q']^{}_N)
\left(\phi^\pm_N([-Q_0]_N^{\frac{1-\mathrm{sgn}([Q]^{}_N)}{2}})
(f([-Q_0]_N^{\frac{1-\mathrm{sgn}([Q]^{}_N)}{2}}[Q]^{}_N))\right)
\quad\textrm{by Definition \ref{definvariant}}\\
&=&\left(\phi^\pm_N([-Q_0]_N^{\frac{1-\mathrm{sgn}([Q']^{}_N[Q]^{}_N)}{2}})
\phi^\pm_N([-Q_0]_N^{\frac{1-\mathrm{sgn}([Q']^{}_N[Q]^{}_N)}{2}}[Q']^{}_N[-Q_0]_N^{\frac{1-\mathrm{sgn}([Q]^{}_N)}{2}})
\right)\\
&&
\left(f([-Q_0]_N^{\frac{1-\mathrm{sgn}([Q]^{}_N)}{2}}[Q]^{}_N))\right)\\
&&\hspace{4cm}\textrm{by the homomorphism property of $\phi^\pm_N$ and
Remark \ref{sgnproduct} (i)}\\
&=&\phi^\pm_N([-Q_0]_N^{\frac{1-\mathrm{sgn}([Q']^{}_N[Q]^{}_N)}{2}})\\
&&\left(f([-Q_0]_N^{\frac{1-\mathrm{sgn}([Q']^{}_N[Q]^{}_N)}{2}}[Q']^{}_N
[-Q_0]_N^{\frac{1-\mathrm{sgn}([Q]^{}_N)}{2}}[-Q_0]_N^{\frac{1-\mathrm{sgn}([Q]^{}_N)}{2}}[Q]^{}_N))\right)\\
&&\hspace{4cm}\textrm{by the fact}~
\mathrm{sgn}([-Q_0]_N^{\frac{1-\mathrm{sgn}([Q']^{}_N[Q]^{}_N)}{2}}[Q']^{}_N[-Q_0]_N^{\frac{1-\mathrm{sgn}([Q]^{}_N)}{2}})=1\\
&&\hspace{4cm}\textrm{obtained from Remarks \ref{sgnproduct} (ii)
and \ref{sgnsgn}, and by Remark \ref{extendedQQ}}\\
&=&\phi_N^\pm([-Q_0]_N^{\frac{1-\mathrm{sgn}([Q']^{}_N[Q]^{}_N)}{2}})
\left(
f([-Q_0]_N^{\frac{1-\mathrm{sgn}([Q']^{}_N[Q]^{}_N)}{2}}[Q']^{}_N[Q]^{}_N)
\right)\quad\textrm{by Remark \ref{sgnproduct} (i)}\\
&=&f([Q']^{}_N[Q]^{}_N)\quad\textrm{by Definition \ref{definvariant}}.
\end{eqnarray*}
\end{proof}

\section {Some subgroups of definite form class groups}

In this last section, we shall find subgroups
of $\mathcal{C}^\pm_N(d_K)$ which are isomorphic to
$\mathrm{Gal}(K_{(N)}/H_K)$,
$\mathrm{Gal}(K_{(N)}/\mathbb{Q}(j(\tau^{}_K)))$ and
$\mathrm{Gal}(K_{(N)}/\mathbb{Q}(\zeta^{}_N))$, respectively.
\par
First, observe that the principal form
$Q_0=x^2+b_Kxy+c_Ky^2$ induces $\omega^{}_{Q_0}=\tau^{}_K$.
Let
\begin{equation*}
\mathcal{C}_{0,\,N}(d_K)=\{[Q_0^\alpha]^{}_N~|~\alpha\in\mathrm{SL}_2(\mathbb{Z})
~\textrm{satisfies}~Q_0^\alpha\in\mathcal{Q}_N(d_K)\}
\quad(\subseteq\mathcal{C}_N(d_K)\subset\mathcal{C}^\pm_N(d_K)).
\end{equation*}

\begin{theorem}\label{overH}
We have
\begin{equation*}
\phi^\pm_N{}^{-1}(\mathrm{Gal}(K_{(N)}/H_K))=\phi_N^{-1}(\mathrm{Gal}(K_{(N)}/H_K))=
\mathcal{C}_{0,\,N}(d_K).
\end{equation*}
\end{theorem}
\begin{proof}
We deduce that for $Q\in\mathcal{Q}_N(d_K)$
\begin{eqnarray*}
[Q]^{}_N\in\phi_N^{-1}(\mathrm{Gal}(K_{(N)}/H_K))
&\Longleftrightarrow&\phi^{}_N([Q]^{}_N)|_{H_K}=\mathrm{id}_{H_K}\\
&\Longleftrightarrow&\phi^{}_N([Q]^{}_N)(j(\tau^{}_K))=j(\tau^{}_K)
\quad\textrm{by (\ref{generationHilbert})}\\
&\Longleftrightarrow&
j(-\overline{\omega}^{}_Q)=j(\tau^{}_K)=j(-\overline{\tau}^{}_K)
\quad\textrm{by Proposition \ref{extended},}\\
&&\textrm{the facts $j(\tau)\in\mathcal{F}_1$ and $-\overline{\tau}^{}_K=\tau^{}_K+b_K$}\\
&\Longleftrightarrow&\gamma(-\overline{\omega}^{}_Q)=-\overline{\tau}^{}_K
\quad\textrm{for some}~\gamma\in\mathrm{SL}_2(\mathbb{Z})~\textrm{by Lemma \ref{jLemma}}\\
&\Longleftrightarrow&\alpha(\omega_Q)=\tau^{}_K\quad\textrm{for some}~\alpha\in\mathrm{SL}_2(\mathbb{Z})\\
&\Longleftrightarrow&Q=Q_0^\alpha\quad\textrm{for some}~\alpha\in\mathrm{SL}_2(\mathbb{Z})
\quad\textrm{by (\ref{wQequal}) and (\ref{wQgamma})}.
\end{eqnarray*}
Hence we conclude that $\phi_N^{-1}(\mathrm{Gal}(K_{(N)}/H_K))=\mathcal{C}_{0,\,N}(d_K)$.
\end{proof}

\begin{remark}
Now, we have shown by Proposition \ref{reciprocity}
and Corollary \ref{overH} that
\begin{equation*}
\mathcal{C}_{0,\,N}(d_K)
~\simeq~W_{K,\,N}/r^{}_N(U_K).
\end{equation*}
Here we shall present an explicit isomorphism
of $\mathcal{C}_{0,\,N}(d_K)$ onto $W_{K,\,N}/r^{}_N(U_K)$ as follows $\colon$
let $[Q_0^\alpha]\in\mathcal{C}_{0,\,N}(d_K)$ with $\alpha
=\begin{bmatrix}A&B\\C&D\end{bmatrix}\in\mathrm{SL}_2(\mathbb{Z})$
satisfying $Q_0^\alpha\in\mathcal{Q}_N(d_K)$.
If we write $Q_0^\alpha=ax^2+bxy+cy^2$, then we get
\begin{equation}\label{abc}
\left\{\begin{array}{ccl}
a&=&A^2+b_KAC+c_KC^2~=~Q_0(A,\,C),\\
b&=&2AB+b_K(AD+BC)+2c_KCD,\\
c&=&B^2+b_KBD+c_KD^2.
\end{array}\right.
\end{equation}
Furthermore, we obtain
by (\ref{wQgamma}) and the fact $-\overline{\tau}^{}_K=\tau^{}_K+b_K$ that
\begin{equation}\label{-omegabar}
-\overline{\omega}^{}_{Q_0^\alpha}
=-\overline{\alpha^{-1}(\omega^{}_{Q_0})}
=-\overline{\alpha^{-1}(\tau^{}_K)}=
\frac{D(-\overline{\tau}^{}_K)+B}{C(-\overline{\tau}^{}_K)+A}
=\begin{bmatrix}D&B\\C&A\end{bmatrix}\begin{bmatrix}1&b_K\\0&1\end{bmatrix}(\tau^{}_K).
\end{equation}
We then achieve by Proposition \ref{extended} that
for any $f\in\mathcal{F}_N$ which is finite at $\tau^{}_K$
\begin{eqnarray*}
\phi^{}_N([Q_0^\alpha]^{}_N)(f(\tau^{}_K))&=&f^{\left[\begin{smallmatrix}
1&\frac{b+b_K}{2}\\0&a
\end{smallmatrix}\right]^{-1}}(-\overline{\omega}^{}_{Q_0^\alpha})
\quad\textrm{by (\ref{+invariant})}\\
&=&f^{\left[\begin{smallmatrix}
1&\frac{b+b_K}{2}\\0&a
\end{smallmatrix}\right]^{-1}}
\left(\begin{bmatrix}D&B\\C&A\end{bmatrix}\begin{bmatrix}1&b_K\\0&1\end{bmatrix}(\tau^{}_K)\right)
\quad\textrm{by (\ref{-omegabar})}\\
&=&f^{
\left[\begin{smallmatrix}1 & \frac{b+b_K}{2}\\0&a
\end{smallmatrix}\right]^{-1}
\left[\begin{smallmatrix}D & B\\C&A
\end{smallmatrix}\right]
\left[\begin{smallmatrix}1 & b_K\\0&1
\end{smallmatrix}\right]}(\tau^{}_K)
\quad\textrm{by (\ref{composition})}\\
&=&f^{Q_0(A,\,C)^{-1}\left[\begin{smallmatrix}
A&-c_KC\\C&A+b_KC
\end{smallmatrix}\right]}(\tau^{}_K)\quad\textrm{by (\ref{abc}) and the fact
$AD-BC=1$}.
\end{eqnarray*}
In the above, $Q_0(A,\,C)^{-1}$ means the inverse of $Q_0(A,\,C)$ in $(\mathbb{Z}/N\mathbb{Z})^\times$.
Observe that
if we let $s=Q_0(A,\,C)^{-1}C$ and $t=Q_0(A,\,C)^{-1}(A+b_KC)$, then
\begin{equation*}
\begin{bmatrix}
t-b_Ks & -c_Ks\\s&t
\end{bmatrix}=
Q_0(A,\,C)^{-1}\begin{bmatrix}A & -c_KC\\C&A+b_KC\end{bmatrix}
\quad(\in W_{K,\,N}).
\end{equation*}
Therefore we establish the desired isomorphism
\begin{equation*}
\begin{array}{ccc}
\mathcal{C}_{0,\,N}(d_K)&\stackrel{\sim}{\rightarrow}&W_{K,\,N}/r^{}_N(U_K)
\vspace{0.1cm}\\
\left[
Q_0^{\left[\begin{smallmatrix}
A&B\\C&D\end{smallmatrix}\right]}\right]^{}_N
& \mapsto&
\left[Q_0(A,\,C)^{-1}\begin{bmatrix}A & -c_KC\\C&A+b_KC\end{bmatrix}\right]
\end{array}
\end{equation*}
by Proposition \ref{reciprocity}.
\end{remark}

\begin{corollary}
We derive
\begin{equation*}
\phi^\pm_N{}^{-1}(\mathrm{Gal}(K_{(N)}/\mathbb{Q}(j(\tau^{}_K))))=
\{[Q_0^\alpha]^{}_N,\,[-Q_0]^{}_N~|~\alpha\in\mathrm{SL}_2(\mathbb{Z})~
\textrm{satisfies}~Q_0^\alpha\in\mathcal{Q}_N(d_K)\}.
\end{equation*}
\end{corollary}
\begin{proof}
We find by Theorem \ref{C+-} and the fact $j(\tau^{}_K)\in\mathbb{R}$ that
\begin{equation*}
\phi^\pm_N([-Q_0]^{}_N)(j(\tau^{}_K))=\mathfrak{c}(j(\tau^{}_K))=j(\tau^{}_K),
\end{equation*}
which shows that $[-Q_0]^{}_N$ is contained in
$\phi^\pm_N{}^{-1}(\mathrm{Gal}(K_{(N)}/\mathbb{Q}(j(\tau^{}_K)))$.
Moreover, since
$[-Q_0]^{}_N$ does not belong to $\mathcal{C}_{0,\,N}(d_K)$ and
\begin{equation*}
[\mathrm{Gal}(K_{(N)}/\mathbb{Q}(j(\tau^{}_K))):
\mathrm{Gal}(K_{(N)}/H_K)]=[K(j(\tau^{}_K)):\mathbb{Q}(j(\tau^{}_K))]=2,
\end{equation*}
we conclude by Theorem \ref{overH} that
\begin{equation*}
\phi^\pm_N{}^{-1}(\mathrm{Gal}(K_{(N)}/\mathbb{Q}(j(\tau^{}_K))))=
\{[Q_0^\alpha]^{}_N,\,[-Q_0]^{}_N~|~\alpha\in\mathrm{SL}_2(\mathbb{Z})~
\textrm{satisfies}~Q_0^\alpha\in\mathcal{Q}_N(d_K)\}.
\end{equation*}
\end{proof}

\begin{theorem}\label{aQ1}
We get
\begin{equation*}
\phi^\pm_N{}^{-1}(\mathrm{Gal}(K_{(N)}/\mathbb{Q}(\zeta^{}_N)))=
\{[Q]^{}_N~|~Q\in\mathcal{Q}^\pm_N(d_K)~
\mathrm{satisfies}~a_Q\equiv1\Mod{N}\}
\end{equation*}
where $Q=a_Qx^2+b_Qxy+c_Qy^2$.
\end{theorem}
\begin{proof}
Observe that for $Q=a_Qx^2+b_Qxy+c_Qy^2\in\mathcal{Q}^\pm_N(d_K)$,
\begin{eqnarray*}
\phi^\pm_N([Q]^{}_N)(\zeta^{}_N)
&=&(\phi^{}_N([\mathrm{sgn}(Q)Q]^{}_N)\,
\mathfrak{c}|_{K_{(N)}}^\frac{1-\mathrm{sgn}(Q)}{2})(\zeta^{}_N)
\quad\textrm{by the definition of $\phi^\pm_N$ described in (\ref{phi+-})}\\
&=&\phi^{}_N([\mathrm{sgn}(Q)Q]^{}_N)(\zeta_N^{\mathrm{sgn}(Q)})\\
&=&\zeta_N^{a_Q^{-1}}\quad\textrm{by Definition \ref{+invariant} and Proposition \ref{extended}}.
\end{eqnarray*}
Here, $a_Q^{-1}$ indicates the inverse of $a_Q$ in $(\mathbb{Z}/N\mathbb{Z})^\times$.
Thus we obtain the restriction homomorphism of $\phi^\pm_N$
\begin{eqnarray*}
c^{}_N~:~\mathcal{C}^\pm_N(d_K) &\rightarrow & \mathrm{Gal}(\mathbb{Q}(\zeta^{}_N)/
\mathbb{Q})\\
\mathrm{[}Q\mathrm{]}^{}_N~~ & \mapsto &
\left(
\zeta^{}_N\mapsto\zeta_N^{a_Q^{-1}}\right),
\end{eqnarray*}
from which we achieve that
\begin{equation*}
\phi^\pm_N{}^{-1}(\mathrm{Gal}(K_{(N)}/\mathbb{Q}(\zeta^{}_N)))=
\mathrm{Ker}(c^{}_N)=
\{[Q]^{}_N~|~\mathcal{Q}\in\mathcal{Q}^\pm_N(d_K)~
\textrm{satisfies}~a_Q\equiv1\Mod{N}\}.
\end{equation*}
\end{proof}

\begin{remark}
Let $N$ and $M$ be positive integers such that $N\,|\,M$.
By Proposition \ref{extended} and the proof of
Theorem \ref{C+-}, we have the commutative diagram in Figure \ref{N|M}
whose first vertical homomorphism is
\begin{equation*}
\begin{array}{ccc}
\mathcal{C}^\pm_M(d_K) & \rightarrow & \mathcal{C}^\pm_N(d_K)\\
\mathrm{[}Q\mathrm{]}^{}_M & \mapsto & [Q]^{}_N.
\end{array}
\end{equation*}
Note that the compositions of horizontal homomorphisms are
$c^{}_M$ and $c^{}_N$, respectively.
Therefore, for each prime $p$
we derive the surjection
\begin{equation*}
\begin{array}{ccc}
\displaystyle\varprojlim_{n\geq1}\mathcal{C}^\pm_{p^n}(d_K) & \rightarrow & \mathbb{Z}_p^\times\\
([Q_1]^{}_{p},\,[Q_2]^{}_{p^2},\,\ldots) & \mapsto &
\left(\displaystyle\lim_{n\rightarrow\infty}a^{}_{Q_n}\right)^{-1}.
\end{array}
\end{equation*}

\begin{figure}[h]
\begin{equation*}
\xymatrixcolsep{2pc}\xymatrix{
~~~\mathcal{C}^\pm_M(d_K)~~~ \ar[r]^{\phi^\pm_M~~~} \ar[d]
& ~~~\mathrm{Gal}(K_{(M)}/\mathbb{Q})~~~ \ar[r]^{\textrm{restriction}}
\ar[d]^{\textrm{restriction}}
& ~~~\mathrm{Gal}(\mathbb{Q}(\zeta^{}_M)/\mathbb{Q})~~~ \ar[d]^{\textrm{restriction}}\\
~~~\mathcal{C}^\pm_N(d_K)~~~ \ar[r]^{\phi^\pm_N~~~}
& ~~~\mathrm{Gal}(K_{(N)}/\mathbb{Q})~~~
\ar[r]^{\textrm{restriction}}
& ~~~\mathrm{Gal}(\mathbb{Q}(\zeta^{}_N)/\mathbb{Q})~~~
}
\end{equation*}
\caption{A commutative diagram of homomorphisms for $N\,|\,M$}\label{N|M}
\end{figure}

\end{remark}

\section*{Acknowledgement}

The authors would like to thank the referee
for all valuable comments and suggestions,
which helped them to improve the earlier version of the manuscript.
\par
The first named author was supported by the
research fund of Dankook university in 2021 and by the
National Research Foundation of Korea(NRF) grant funded by the Korea government(MSIT) (No. 2020R1F1A1A01073055).
The third named author was supported
by Hankuk University of Foreign Studies Research Fund of 2021 and by the National Research Foundation of Korea(NRF) grant funded by the Korea government(MSIT) (No. 2020R1F1A1A01048633).
The fourth named author was supported by the National Research Foundation of Korea(NRF) grant funded by the Korea government(MSIT) (No. 2020R1C1C1A01006139).

\bibliographystyle{amsplain}

\address{
Department of Mathematics\\
Dankook University\\
Cheonan-si, Chungnam 31116\\
Republic of Korea} {hoyunjung@dankook.ac.kr}
\address{
Department of Mathematical Sciences \\
KAIST \\
Daejeon 34141\\
Republic of Korea} {jkgoo@kaist.ac.kr}
\address{
Department of Mathematics\\
Hankuk University of Foreign Studies\\
Yongin-si, Gyeonggi-do 17035\\
Republic of Korea} {dhshin@hufs.ac.kr}
\address{
Department of Mathematics Education\\
Pusan National University\\
Busan 46241\\Republic of Korea}
{dsyoon@pusan.ac.kr}

\end{document}